\documentclass[11pt,a4paper]{article}   	
\usepackage{geometry}                		
\usepackage{amsmath,amssymb,amsthm}
\usepackage{graphicx}				
\usepackage{braket,amsfonts}								
\usepackage{authblk,bm}
\usepackage{cleveref}
\usepackage{amssymb}
\usepackage[medium,compact]{titlesec}
\usepackage{braket,amsfonts}
\newtheorem{lemma}{Lemma}
\newtheorem{theorem}{Theorem}
\usepackage{amsmath}
\newcommand{\bmv}{\boldsymbol{v}}
\newcommand{\bmu}{\boldsymbol{u}}
\newcommand{\ii}{\mathrm{i}}

\newcommand{\ee}{\mathrm{e}}
\newtheorem{problem}{Problem}
\newtheorem{remark}{Remark}

\graphicspath{{Figure/}}
\theoremstyle{definition}
\newtheorem{example}{Example}
\usepackage{array}
\usepackage{subfigure, breqn}

\usepackage{pgfplots}

\newtheorem{assumption}{Assumption}

\usepackage{algorithmic}
\usepackage{graphicx, epstopdf}

\title{A novel Newton method for inverse elastic scattering problems}
\author{Yan Chang, \thanks{School of Mathematics, Harbin Institute of Technology, Harbin, P. R. China. ({21B312002@stu.hit.edu.cn}).}
            Yukun Guo, \thanks{School of Mathematics, Harbin Institute of Technology, Harbin, P. R. China. ({ykguo@hit.edu.cn}, Corresponding author).}
	   Hongyu Liu,\thanks{Department of Mathematics, City University of Hong Kong, Hong Kong SAR, China. ({hongyliu@cityu.edu.hk}).}
	  and Deyue Zhang\thanks{School of Mathematics, Jilin University, Changchun, P. R. China.  ({dyzhang@jlu.edu.cn}).}}
		\date{}							

\begin{document}
\maketitle

\begin{abstract}
This work is concerned with an inverse elastic scattering problem of identifying the unknown rigid obstacle embedded in an open space filled with a homogeneous and isotropic elastic medium. A Newton-type iteration method relying on the boundary condition is designed to identify the boundary curve of the obstacle. Based on the Helmholtz decomposition and the Fourier-Bessel expansion, we explicitly derive the approximate scattered field and its derivative on each iterative curve. Rigorous mathematical justifications for the proposed method are provided. Numerical examples are presented to verify the effectiveness of the proposed method.
\end{abstract}

\textbf{Keywords:}
inverse scattering, elastic wave, Newton method, convergence

	

\section{Introduction}

As an important noninvasive testing technology for interrogating material properties, the inverse elastic scattering problem arises in diverse areas such as material science, non-destructive testing, medical imaging, seismology, and many other application areas (see, e.g., \cite{Landau, Rose}).
In this work, we are concerned with the inverse obstacle problem of determining the boundary of an unknown elastically rigid obstacle embedded in a homogeneous and isotropic elastic medium with a unit mass density from noisy measurements of the scattered elastic field corresponding to several incident fields impinged on the obstacle.

Suppose a bounded domain $D$ with $C^2$ boundary $\partial D$ in the homogeneous background space $\mathbb{R}^2$ is occupied by an elastically rigid obstacle. Given the source point $z\in\mathbb{R}^2\backslash\overline{D}$ and the polarization $\boldsymbol{p}\in\mathbb{S}:=\{x\in\mathbb{R}^2:|x|=1\},$ the obstacle scattering problem is described by the following boundary value problem
\begin{align}\label{eq:bvp}
	\left\{
	\begin{aligned}
		&\Delta^*\boldsymbol{u}+\omega^2\boldsymbol{u}=-\delta_z(x)\boldsymbol{p},\quad \text{in }\mathbb{R}^2\backslash\overline{D},\\
		&\boldsymbol{u}=0,\quad \text{on }\partial D,
	\end{aligned}
	\right.
\end{align}
where the vectorial field $\boldsymbol{u}$ denotes the displacement of the total field, $\omega>0$ is the angular frequency, and the Lam\'{e} operator $\Delta^*$ is given by
$$
\Delta^*\boldsymbol{u}:=\mu\Delta\boldsymbol{u}+(\lambda+\mu)\nabla\nabla\cdot\boldsymbol{u},
$$  
with Lam\'{e} constants $\lambda,\mu$ satisfying $\mu>0,\lambda+\mu>0.$

Outside $D,$ the solution to \eqref{eq:bvp} has the Helmholtz decomposition 
\begin{align*}
	\boldsymbol{u}=\boldsymbol{u}_p+\boldsymbol{u}_s,
\end{align*}
where the compressional part $\boldsymbol{u}_p$ and the shear part $\boldsymbol{u}_s$ are respectively given by
$$
\boldsymbol{u}_p=-\dfrac{1}{k_p^2}\text{grad}\,\text{div}\,\boldsymbol{u},\quad \boldsymbol{u}_s=-\dfrac{1}{k_s^2}\text{grad}^\bot\,\text{div}^\bot\,\boldsymbol{u},
$$
with $k_p=\omega/\sqrt{\lambda+2\mu},$ $k_s=\omega/\sqrt{\mu}$ and $\text{grad}^\bot:=(-\partial_2,\partial _1)^\top,\quad\text{div}^\bot:=(-\partial_2,\partial_1)$. In addition, the Kupradze-Sommerfeld radiation conditions \cite{Kupradze} should be imposed on $\bm{u}_p$ and $\bm{u}_s$ such that
\begin{align}\label{eq:KS}
	\lim\limits_{|x|\to\infty}\sqrt{|x|}\left(\dfrac{\partial\boldsymbol{u}_p}{\partial |x|}-\mathrm{i}k_p\boldsymbol{u}_p\right)=0,\quad
	\lim\limits_{|x|\to\infty}\sqrt{|x|}\left(\dfrac{\partial\boldsymbol{u}_s}{\partial |x|}-\mathrm{i}k_s\boldsymbol{u}_s\right)=0.
\end{align}

It has been proven in \cite{Bramble} that there exists a unique solution $\boldsymbol{v}=\boldsymbol{u}-\boldsymbol{u}^i\in \left(H_{\text{loc}}^1(\mathbb{R}^2\backslash\overline{D})\right)^2$ to \eqref{eq:bvp}--\eqref{eq:KS}. Here, the incident field $\boldsymbol{u}^i$ can be explicitly given by
\begin{align*}
	\bm{u}^i=\bm{u}^i(x, z;\bm{p})=\mathbb{G}(x, z)\bm{p},\quad x\in\mathbb{R}^2\backslash\left(\overline{D}\cup\{z\}\right),
\end{align*}
with $\mathbb{G}(x; z)$ being the fundamental solution to the Navier equation.

In this paper, we are interested in the inverse obstacle scattering problem to reconstruct the obstacle $D$ from the scattered field $\boldsymbol{v}$ corresponding to several incident waves $\boldsymbol{u}^i.$
Over the last few years, inverse elastic scattering problems have been paid enduring attention and a large number of significant computational methods have been developed in the vast literature.
Roughly speaking, they can be divided into two categories: imaging-based direct methods and nonlinear optimization-based iterative methods.
The basic idea of the former qualitative methods (including the linear sampling method \cite{Arens, Chara}, the factorization method \cite{chara2}, the reverse time migration method \cite{chen}, and the direct sampling method \cite{liu}) is to compute some indicator function over the sampling domain and justify whether a sampling point lies inside or outside the obstacle by the indicator values. Such methods require very weak a priori information concerning the targets and are computationally efficient, but the reconstruction is less accurate. 

Meanwhile, the quantitative methods usually serve the purpose of approximating the boundary curve of the obstacle from physical motivation with a specific mathematical description. The quantitative methods provide a relatively accurate reconstruction while may be computationally intensive and a good initial guess is usually required. Recently, a variety of quantitative methods have been established to deal with inverse elastic scattering problems. In \cite{BaoHuSunYin}, Bao et al. design a descent algorithm for recovering the penetrable anisotropic elastic body embedded in a homogeneous isotropic background medium. They employ a descent method to find the parameters of the unknown surface in a finite-dimensional space and the correctness of the parameters needs to be evaluated at each iteration step. In \cite{Dong}, Dong et al. propose a system of nonlinear integral equations and develop two iterative reconstruction methods for determining the location and shape of a rigid obstacle from phased or phaseless far-field data for a single incident plane wave.
In \cite{YueLiLiYuan}, Yue et al. propose a frequency recursive approach based on the domain derivative to solve the inverse elastic scattering problem numerically. We also refer to \cite{LiWangWangZhao16} for a continuation method to solve the inverse elastic scattering problem. Though there is a fertile research background for the iterative-based method, a main drawback of most existing work is that a sequence of direct and adjoint scattering problems are involved at each iteration and thus they are usually computation-consuming.

To develop a quantitative method with less computational effort, we shall focus on a novel numerical scheme by coupling the decomposition method and Newton's iterative method. The main idea of our method is to overcome the ill-posedness and the nonlinearity through the Fourier-Bessel approximation and the linearization, respectively. Since there is no forward solver involved, our method is easy to implement and convenient. In fact, in very recent years, the simple Newton methods have been investigated in the literature. 
In \cite{JCP}, the authors propose a Newton-type method for the inverse elastic obstacle scattering problems. Based on the layer potential representations on an auxiliary curve inside the obstacle, the scattered field together with its derivative on each iteration surface can be easily derived and there is no forward solver involved in the iteration. Similar ideas can be also found in inverse acoustic scattering problems. For instance, \cite{era} presents a hybrid method for the interior inverse scattering problem. In addition, the interior eigenfunctions are integrated into the Newton method to formulate an imaging scheme for identifying obstacles from the far-field data \cite{He}.


To give a convergence analysis of the proposed iteration method mathematically, we extend the Fourier-Bessel approximation developed in \cite{IPI22} to the elastic wave case and give a rigorous mathematical analysis of the approximation.
Once the scattered field at each iterative curve is approximated, we can determine the boundary of the obstacle as the locations where the Dirichlet condition is satisfied. Noticing that the total field on the iterative curves is nonlinear concerning the curve, we adopt the iterative strategy to solve the nonlinear equation. Compared with the Newton method based on the layer potential representation \cite{JCP}, the method developed in the current work is much more direct due to the fact that the scattered field is represented by the Fourier expansion and the Fourier coefficients can be easily computed due to the orthogonal base. In other words, once the scattered data is acquired on some measurement curve, the approximate scattered field at any position can be obtained immediately by some easy inner products and basic operation. In this sense, our method is straightforward.
In addition, this approximate property plays a key role in the convergence analysis, and we mathematically point out that the iterative curve converges to the exact boundary when the measurement noise tends to zero.

The rest of this paper is organized as follows: In the next section, we give a more specified description of the inverse elastic scattering problem. In Section \ref{sec:FB}, we establish the Fourier-Bessel approximation to the scattered field based on the Helmholtz decomposition technique and carry on the mathematical analysis. In Section \ref{sec:convergence}, we propose a Newton-type method and establish the corresponding convergence analysis. Numerical experiments are conducted in Section \ref{sec:example} to verify the performance of the proposed method. Finally, this paper is concluded with some remarks in Section \ref{sec:conclusion}.
	

\section{Problem setting}
We proceed with a more detailed description of the inverse elastic scattering problem under consideration.
Let $D\subset\mathbb{R}^2$ be a simply connected domain with a $C^2$ boundary $\partial D.$
Assume that $D\subset B_\rho:=\{x\in\mathbb{R}^2:|x|=\rho\}$ such that $B_\rho\backslash\overline{D}$ is connected. Denote $\Gamma_\rho=\partial B_\rho.$  

Denote the fundamental displacement tensor to the Navier equation by
\begin{equation*}
	\mathbb{G}(x, y):=\frac{1}{\mu}\Phi_s(x,y)\mathbb{I}+\frac{1}{\omega^2}\nabla_x\nabla_x^\top\left(\Phi_s(x,y)-\Phi_p(x,y)\right),
\end{equation*}
where $\mathbb{I}$ is the $2\times 2$ identity matrix, and $\Phi_\xi(x, y)$ denotes the fundamental solution to the Helmholtz equation with wavenumber $k_\xi,$ i.e.,
\begin{align*}
	\Phi_\xi(x, y)=\dfrac{\mathrm{i}}{4}H_0^{(1)}(k_\xi|x-y|),\quad \xi\in\{p,s\}.
\end{align*}
Here and in the following, $H_n^{(1)}$ is the Hankel function of the first kind of order $n.$ It is easy to check that the incident field $\boldsymbol{u}^i(x; z,\boldsymbol{p})$ satisfies 
\begin{align}\label{eq:uieq}
	\Delta^*\boldsymbol{u}^i+\omega^2\boldsymbol{u}^i=-\delta_z(x)\boldsymbol{p}.
\end{align}
Combining \eqref{eq:bvp} and \eqref{eq:uieq},  we derive that the displacement of the scattered field $\bmv=\bmu-\bmu^i$ is then governed by
\begin{equation}\label{eq:boundaryvalueproblem}
	\begin{cases}
		\Delta^*\bm{v}+\omega^2\bm{v}=0, &\text{in }\mathbb{R}^2\backslash\overline{D},\\
		\hfill\bm{u}=0, & \text{on }\partial D.
	\end{cases}
\end{equation}
In addition, the scattered field $\boldsymbol{v}$ is required to satisfy the Kupradze-Sommerfeld radiation condition
\begin{align*}
	\lim\limits_{r\to\infty}\sqrt{r}(\partial_r\boldsymbol{v}_p-\mathrm{i}k_p\boldsymbol{v}_p)=0,
	\quad\lim\limits_{r\to\infty}\sqrt{r}(\partial_r\boldsymbol{v}_s-\mathrm{i}k_s\boldsymbol{v}_s)=0,
	\quad r=|x|.
\end{align*}
Here, the compressional and shear wave components of $\bm{v}$ are respectively given by
\begin{equation*}
	\bm{v}_p=-\dfrac{1}{k_p^2}\text{grad}\,\text{div}\,\bm{v},\quad 
	\bm{v}_s=-\dfrac{1}{k_s^2}\text{grad}^\bot\,\text{div}^\bot\,\bm{v}.
\end{equation*}

With these preparations, the inverse scattering problem under consideration can be stated as follows:
\begin{problem}[Inverse elastic scattering problem]\label{problem} 
For fixed parameters $\omega,\lambda,\mu$ and $\bm{p}$, given the scattered data $\{\bm{v}(x; z): x\in\Gamma_\rho, z\in\Gamma_\rho\}$
determine the boundary $\partial D$ of the obstacle.
\end{problem}
	

\section{Fourier approximation}\label{sec:FB}
In this section, we shall approximate the scattered field outside the obstacle through the Fourier expansion and then establish several estimates with regard to this Fourier approximation.

We first recall that for any solution $\boldsymbol{v}$ of the elastic wave equation \eqref{eq:boundaryvalueproblem} in $\mathbb{R}^2\backslash\overline{D}$, the Helmholtz decomposition holds
\begin{align}\label{eq:HD}
	\boldsymbol{v}=\nabla{\phi}_p+\mathbf{curl}{\phi}_s,
\end{align}
where ${\phi}_p$ and ${\phi}_s$ are scalar potential functions. By substituting  \eqref{eq:HD} into \eqref{eq:boundaryvalueproblem}, we obtain that
\[
\nabla[(\lambda+2\mu)\Delta{\phi}_p+\omega^2{\phi}_p]+\mathbf{curl}(\mu\Delta{\phi}_s+\omega^2{\phi}_s)=0,
\]
which can be fulfilled once ${\phi}_\xi$ satisfies the Helmholtz equation
\begin{equation}\label{eq:phi_psi2D}
	\Delta {\phi}_\xi+k_\xi^2{\phi}_\xi=0, \quad\xi\in\{p, s\}.
\end{equation}
In addition, ${\phi}_\xi,\,\xi\in\{p, s\}$ are required to satisfy the Sommerfeld radiation condition
\begin{equation*}
	\lim\limits_{r\to\infty}\sqrt{r}(\partial_r{\phi}_\xi-\mathrm{i}k_\xi{\phi}_\xi)=0.
\end{equation*}


\subsection{Fourier expansion}

Take $R>0$ such that $B_R:=\{x\in\mathbb{R}^2:|x|<R\}\subset D.$ Under the polar coordinate $(r, \theta):\   x=(r\cos\theta, r\sin\theta),$ the solution to \eqref{eq:phi_psi2D} outside $B_R$ can be expanded by
\begin{align}\label{eq:phi_psi_expansion}
	\phi_\xi=\sum_{n=-\infty}^{\infty}\dfrac{H_n^{(1)}(k_\xi r)}{H_n^{(1)}(k_\xi R)}\hat{\phi}_{\xi, n}\ee^{\ii n\theta},\quad \xi\in\{p, s\},
\end{align}
where $\hat{\phi}_{\xi, n}$ are the Fourier coefficients. Let $\{\bm{e}_r,\bm{e}_\theta\}$ be the local orthonormal basis given by
\[
\bm{e}_r=(\cos\theta,\sin\theta)^\top,\ \ 
\bm{e}_\theta=(-\sin\theta,\cos\theta)^\top,
\]
then it is easy to verify that, for any function $w,$
\[
\nabla w=\partial_r w\bm{e}_r+\dfrac{1}{r}\partial_\theta w\bm{e}_\theta,\quad  
\mathbf{curl} w=\dfrac{1}{r}\partial_\theta w\bm{e}_r-\partial_r w\bm{e}_\theta.
\]
Hence, under the polar coordinates, the Helmholtz decomposition \eqref{eq:HD} reads
\begin{align}\label{eq:HDnew}
	\boldsymbol{v}=\left(\partial_r {\phi}_p+\dfrac{1}{r}\partial_\theta{\phi}_s\right)\bm{e}_r+\left(\dfrac{1}{r}\partial_\theta {\phi}_p-\partial_r{\phi}_s\right)\bm{e}_\theta.
\end{align}

Furthermore, let
\begin{align}\label{eq:UnVn}
	\bm{U}_n(\theta):=\ee^{\mathrm{i}n\theta}\bm{e}_r,\ \
	\bm{V}_n(\theta):=\ee^{\mathrm{i}n\theta}\bm{e}_\theta,
\end{align}
and the orthogonality is immediate in the sense that
$$
\int_0^{2\pi}\bm{U}_m(\theta)\cdot\overline{\bm{U}_n(\theta)}\mathrm{d}\theta =\int_0^{2\pi}\bm{V}_m(\theta)\cdot\overline{\bm{V}_n(\theta)}\mathrm{d}\theta = 2\pi\delta_{mn},\quad\int_0^{2\pi}\bm{U}_m(\theta)\cdot\overline{\bm{V}_n(\theta)}\mathrm{d}\theta = 0,
$$
where the overbar denotes the complex conjugate and $\delta_{mn}$ is the Kronecker Delta.

Combining \eqref{eq:phi_psi_expansion}, \eqref{eq:HDnew} and \eqref{eq:UnVn}, we have
\begin{align}\nonumber
	\bm{v}(r,\theta)&=\sum_{n=-\infty}^{\infty}\left[\left(\alpha_{p, n}(r)\hat{\phi}_{p, n}+\dfrac{\mathrm{i}n}{r}\beta_{s, n}(r)\hat{\phi}_{s, n}\right)\bm{U}_n(\theta)\right.\\
	&\quad\left.+\left(\dfrac{\mathrm{i}n}{r}\beta_{p, n}(r)\hat{\phi}_{p, n}-\alpha_{s, n}(r)
	\hat{\phi}_{s, n}\right)\bm{V}_n(\theta)\right],\label{eq:us}
\end{align}
where 
\begin{equation*}
	\alpha_{\xi, n}(r):=\dfrac{k_\xi{H_n^{(1)}}'(k_\xi r)}{H_n^{(1)}(k_\xi R)},\ \ 
	\beta_{\xi, n}(r):=\dfrac{H_n^{(1)}(k_\xi r)}{H_n^{(1)}(k_\xi R)},
	\quad  \xi=p, s.
\end{equation*}

Taking $r=\rho$ in \eqref{eq:us} and we denote, for the sake of notational simplicity, by $\alpha_{\xi, n}:=\alpha_{\xi, n}(\rho)$ and $\beta_{\xi, n}:=\beta_{\xi, n}(\rho)$ for $\xi\in\{p, s\}$ and a fixed $\rho$ throughout the paper. Without loss of generality, we also assume that $H_n^{(1)}(k_\xi \rho)\ne 0,\xi\in\{p, s\}$. 

Multiplying equation \eqref{eq:us} respectively by $\bm{U}_n$ and $\bm{V}_n,$  and integrating over $\Gamma_\rho$, we derive the 2-by-2 linear system for the Fourier coefficients
\begin{equation}\label{eq:phipsiEq}
	\begin{bmatrix}
		\alpha_{p, n} & \dfrac{\mathrm{i}n}{\rho}\beta_{s, n}\\
		\dfrac{\mathrm{i}n}{\rho}\beta_{p, n} & -\alpha_{s, n}
	\end{bmatrix}
	\begin{bmatrix}
		\hat{\phi}_{p, n} \\
		\hat{\phi}_{s, n}
	\end{bmatrix}
	=\begin{bmatrix}
		f_{p, n} \\
		f_{s, n}
	\end{bmatrix},
\end{equation}
where 
\begin{align*}
	f_{p, n}=\dfrac{1}{2\pi\rho}\int_{\Gamma_\rho}\boldsymbol{v}\cdot\overline{\boldsymbol{U}_n}\mathrm{d}s,\ \
	f_{s, n}=\dfrac{1}{2\pi\rho}\int_{\Gamma_\rho}\boldsymbol{v}\cdot\overline{\boldsymbol{V}_n}\mathrm{d}s,
\end{align*}

The solution to \eqref{eq:phipsiEq} is given by
\begin{align}\label{eq:sol}
	\begin{bmatrix}
		\hat{\phi}_{p, n}\\
		\hat{\phi}_{s, n}
	\end{bmatrix}
	=\frac{1}{\Lambda_n\beta_{p, n}\beta_{s, n}}
	\begin{bmatrix}
		-\alpha_{s, n}f_{p, n}-\frac{\ii n}{\rho}\beta_{s, n}f_{s, n}\\
		-\frac{\ii n}{\rho}\beta_{p, n}f_{p, n}+\alpha_{p, n}f_{s, n}
	\end{bmatrix}.
\end{align}
where (see \cite{LiWangWangZhao16})
\[
\Lambda_n=\Lambda_n(\rho):=\dfrac{n^2}{\rho^2}-\dfrac{\alpha_{p, n}\alpha_{s, n}}{\beta_{p, n}\beta_{s, n}}\ne 0,\quad  \forall n\in\mathbb{Z}.
\]

Further, we approximate the scattered field $\bm{v}$ by the truncated Fourier expansion: 
\begin{align}\nonumber
	\boldsymbol{v}_N(x)&=\sum_{n=-N}^{N} \left[\left(\alpha_{p, n}(r)\hat{\phi}_{p, n}+\dfrac{\mathrm{i}n}{r}\beta_{s, n}(r)\hat{\phi}_{s, n}\right)\boldsymbol{U}_n(\theta)\right.\\
	&\quad \left.+\label{eq:us_a}
	\left(\dfrac{\mathrm{i}n}{r}\beta_{p, n}(r)\hat{\phi}_{p, n}-\alpha_{s, n}(r)
	\hat{\phi}_{s, n}\right)\boldsymbol{V}_n(\theta)\right],
\end{align}
with $r=|x|$. Taking the ill-posedness into consideration, let $\boldsymbol{v}^\delta\in \left(L^2(\Gamma_\rho)\right)^2$ be the noisy measurement of the scattered data satisfying
\begin{equation}\label{eq:noisyData}
	\left\|\boldsymbol{v}-\boldsymbol{v}^\delta\right\|_{\left(L^2(\Gamma_\rho)\right)^2}\le\delta
	\left\|\boldsymbol{v}\right\|_{\left(L^2(\Gamma_\rho)\right)^2},
\end{equation}
with $\delta\in(0, 1)$ the noise level. Using the noisy measurements, we rewrite the approximation of the scattered field as the truncated expansion
\begin{align}\nonumber
	\boldsymbol{v}^\delta_N(x)&=\sum_{n=-N}^{N}\left[\left(\alpha_{p, n}(r)\hat{\phi}_{p, n}^\delta
	+\dfrac{\mathrm{i}n}{r}\beta_{s, n}(r)\hat{\phi}_{s, n}^\delta\right)\boldsymbol{U}_n\right.\\
	&\quad\left.+
	\left(\dfrac{\mathrm{i}n}{r}\beta_{p, n}(r)\hat{\phi}_{p, n}^\delta-\alpha_{s, n}(r)
	\hat{\phi}_{s, n}^\delta\right)\boldsymbol{V}_n\right],\label{eq:us_a_noisy}
\end{align}
where $\hat{\phi}_{p, n}^\delta$ and $\hat{\phi}_{s, n}^\delta$ can be derived similarly to \eqref{eq:phipsiEq} by replacing the right hand side in \eqref{eq:phipsiEq} with 
$$
\begin{bmatrix}
	f_{p, n}^\delta \\
	f_{s, n}^\delta 
\end{bmatrix}
=\dfrac{1}{2\pi\rho}
\begin{bmatrix}
\displaystyle\int_{\Gamma_\rho}\boldsymbol{v}^\delta\cdot\overline{\boldsymbol{U}_n}\mathrm{d}s\\
\displaystyle\int_{\Gamma_\rho}\boldsymbol{v}^\delta\cdot\overline{\boldsymbol{V}_n}\mathrm{d}s
\end{bmatrix}.
$$

\subsection{Error estimates}

Before heading for the error estimate, we shall recall several properties concerning the Hankel functions $H_n^{(1)}(t).$
\begin{lemma}\cite[Lemma~3.3]{IP15}
	\label{lem:IP15_lem3.3}
	Let $n\in\mathbb{Z}$ and $0<t_1\le t_2.$ Then we have
	\begin{align}\label{eq:Hn1_property}
		& \left|H_n^{(1)}(t_2)\right|\le\left|H_n^{(1)}(t_1)\right|,\\
		& \label{eq:Hn1_Deriv_property} \left|{H_n^{(1)}}'(t_2)\right|\le\left(1+\dfrac{|n|}{t_2}\right)\left|H_n^{(1)}(t_2)\right|.
	\end{align}
\end{lemma}
\begin{lemma}\cite[Lemma~2.3]{IPI22}
	\label{lem:IPI}
	For $t>0$ and $n\in\mathbb{N}$ such that $n>(\ee t+1)/2,$ we have
	\begin{align}\label{eq:H_n1_estimate}
		\dfrac{1}{2}\le\dfrac{\pi t^n\left|H_n^{(1)}(t)\right|}{3\cdot 2^{n-1}\Gamma(n)}\le\ee^t,
	\end{align}
	where $\Gamma(\cdot)$ denotes the Gamma function. 
\end{lemma}

Throughout the paper, we assume that
\begin{align}\label{eq:assume}
	N\ge N_0:=4k_s\rho+1,
\end{align}
which is proposed for later convergence analysis.

Based on these lemmas, we obtain the following approximation results.
\begin{theorem}
	Under the Assumption \eqref{eq:assume}, there exist positive constants $C_1$ and $C_2$ independent of $N$ such that 
	\begin{equation*}
		\|\boldsymbol{v}-\boldsymbol{v}_N\|_{\left(L^2(\Gamma_\rho)\right)^2}^2
		\le C_1\tau_1^{-2N}+C_2N^2\tau_1^{4-2N},
	\end{equation*}
	where $\tau_1=\rho/R.$
\end{theorem}

\begin{proof}
	From \eqref{eq:us} and \eqref{eq:us_a}, it can be seen that
	\begin{align*}
		&\quad\dfrac{1}{2\pi\rho}\left\|\boldsymbol{v}-\boldsymbol{v}_N\right\|^2_{\left(L^2(\Gamma_\rho)\right)^2}
		=\dfrac{1}{2\pi\rho}\int_{\Gamma_\rho}\left|\boldsymbol{v}-\boldsymbol{v}_N\right|^2\mathrm{d}s(x)\\
		&=\sum_{|n|>N}\left(\left|\alpha_{p, n}\hat{\phi}_{p, n}+\dfrac{\mathrm{i}n}{\rho}\beta_{s, n}\hat{\phi}_{s, n}\right|^2
		+\left|\dfrac{\mathrm{i}n}{\rho}\beta_{p, n}\hat{\phi}_{p, n}-\alpha_{s, n}\hat{\phi}_{s, n}\right|^2\right)\\
		&\le 2\sum_{\xi\in\{p, s\}}\sum_{|n|>N}\left(\left|\alpha_{\xi, n}\right|^2+\dfrac{n^2}{\rho^2}\left|\beta_{\xi, n}\right|^2\right)|\hat{\phi}_{\xi, n}|^2.
	\end{align*}
	
	Noticing \Cref{lem:IP15_lem3.3,lem:IPI}, it holds that for $\xi\in\{p, s\},$
	\begin{align*}
		|\beta_{\xi, n}|\le 2\ee^{k_\xi\rho}\left(\dfrac{R}{\rho}\right)^{|n|},\ \
		|\alpha_{\xi, n}|\le\left(k_\xi+\dfrac{|n|}{\rho}\right)
		\left|\beta_{\xi, n}\right|\le 2\left(k_\xi+\dfrac{|n|}{\rho}\right)\ee^{k_\xi\rho}\left(\dfrac{R}{\rho}\right)^{|n|},
	\end{align*}
	which enables us to proceed with the proof by
	\begin{align}
		\nonumber&\quad\dfrac{1}{2\pi\rho}\left\|\bm{v}-\bm{v}_N\right\|^2_{\left(L^2(\Gamma_\rho)\right)^2}\\
		\nonumber&\le 2\sum_{\xi\in\{p, s\}}\sum_{|n|>N}\left(8k_\xi^2+\dfrac{12n^2}{\rho^2}\right)\ee^{2k_\xi\rho}
		\tau_1^{-2|n|}|\hat{\phi}_{\xi, n}|^2\\
		&\le 32k_s^2\ee^{2k_s\rho}\sum_{\xi\in\{p, s\}}\sum_{|n|>N}\tau_1^{-2|n|}
		|\hat{\phi}_{\xi, n}|^2+\dfrac{48\ee^{2k_s\rho}}{\rho^2}
		\sum_{\xi\in\{p, s\}}\sum_{|n|>N}n^2\tau_1^{-2|n|}|\hat{\phi}_{\xi, n}|^2.\label{eq:v_Error_deduce}
	\end{align}
	Here, 
	\begin{align}
		\nonumber\sum_{\xi\in\{p, s\}}\sum_{|n|>N}\tau_1^{-2|n|}|\hat{\phi}_{\xi, n}|^2
		&\le\left(\sum_{|n|>N}\tau_1^{-2|n|}\right)\left(\sum_{\xi\in\{p, s\}}\sum_{|n|>N}|\hat{\phi}_{\xi, n}|^2\right)\\
		&\le\dfrac{1}{\pi R}\tau_1^{-2N}(\tau_1^2-1)^{-1}\left(\|{\phi}_p\|_{L^2(\Gamma_R)}^2+
		\|{\phi}_s\|_{L^2(\Gamma_R)}^2\right),\label{eq:I1}
	\end{align}
	and 
	\begin{align}
		\nonumber \sum_{|n|>N}n^2\tau_1^{-2|n|}\Big(|\hat{\phi}_{p, n}|^2+|\hat{\phi}_{s, n}|^2\Big)
		&\le\left(\sum_{|n|>N}n^2\tau_1^{-2|n|}\right)\left(\sum_{\xi\in\{p, s\}}\sum_{|n|>N}|\hat{\phi}_{\xi, n}|^2\right)\\
		&\le \dfrac{4N^2}{\pi R}\tau_1^{4-2N}(\tau_1^2-1)^{-3}\left(\|{\phi}_p\|_{L^2(\Gamma_R)}^2+
		\|{\phi}_s\|_{L^2(\Gamma_R)}^2\right),\label{eq:I2}
	\end{align}
	where we have used the fact that (see the Appendix)
	\begin{equation*}
		\sum_{|n|>N}n^2\tau_1^{-2|n|}\le 8 N^2\tau_1^{4-2N}(\tau_1^2-1)^{-3}.
	\end{equation*}
	
	Combining \eqref{eq:v_Error_deduce}, \eqref{eq:I1} and \eqref{eq:I2} gives
	\begin{equation*}
		\|\bm{v}-\bm{v}_N\|_{\left(L^2(\Gamma_\rho)\right)^2}^2
		\le C_1\tau_1^{-2N}+C_2N^2\tau_1^{4-2N},
	\end{equation*}
	where
	\begin{align*}
		C_1 & =\dfrac{32k_s^2\ee^{2k_s\rho}}{\pi R(\tau_1^2-1)}\left(\|{\phi}_p\|_{L^2(\Gamma_R)}^2+
		\|{\phi}_s\|_{L^2(\Gamma_R)}^2\right),\\
		C_2&=\dfrac{192\ee^{2k_s\rho}}{\pi \rho^2 R(\tau_1^2-1)^3}\left(\|{\phi}_p\|_{L^2(\Gamma_R)}^2+
		\|{\phi}_s\|_{L^2(\Gamma_R)}^2\right).
	\end{align*}
	
\end{proof}

\begin{theorem}\label{thm:error1}
		Under Assumption \eqref{eq:assume}, there exist constants $C_3,C_4>0$ such that 
		\begin{equation*}
			\left\|\bmv_N-\bmv_N^\delta\right\|^2_{\left(L^2(\Gamma_\rho)\right)^2}\le C_3\delta^2+C_4 N^5\delta^2. \label{eq:v_noise_error_delta}
		\end{equation*}
\end{theorem}

\begin{proof}
		With the noisy scattered data \eqref{eq:noisyData}, it holds
				\begin{equation}\label{eq:rhs_estimate1}
			\left|f_{p, n}-f_{p, n}^\delta\right|=\dfrac{1}{2\pi\rho}\left|\int_{\Gamma_\rho}(\bm{v}-\bm{v}^\delta)
			\cdot\overline{\boldsymbol{U}_n}\mathrm{d}s\right|
			\le\dfrac{\|\bm{v}-\bm{v}^\delta\|_{(L^2(\Gamma_\rho))^2}}{\sqrt{2\pi\rho}}
			\le\dfrac{\delta\|\bm{v}\|_{(L^2(\Gamma_\rho))^2}}{\sqrt{2\pi\rho}}.
		\end{equation}
		Similarly,
		\begin{align}\label{eq:rhs_estimate2}
			\left|f_{s, n}-f_{s, n}^\delta\right|\le\dfrac{\delta\|\bm{v}\|_{\left(L^2(\Gamma_\rho)\right)^2}}{\sqrt{2\pi\rho}}.
		\end{align}
		Noticing \eqref{eq:sol}, we derive that
		\begin{align}\label{eq:phi_error}
			&\left|\hat{\phi}_{p, n}-\hat{\phi}_{p, n}^\delta\right|\le \frac{1}{|\Lambda_n\beta_{p, n}|}\left(\left|\dfrac{\alpha_{s, n}}{\beta_{s, n}}\right|
			\left|f_{p, n}-f_{p, n}^\delta\right|+\dfrac{|n|}{\rho}
			\left|f_{s, n}-f_{s, n}^\delta\right|\right),\\
			\label{eq:psi_error}
			&\left|\hat{\phi}_{s, n}-\hat{\phi}_{s, n}^\delta\right|\le 
			\frac{1}{|\Lambda_n\beta_{s, n}|}\left(\dfrac{|n|}{\rho}\left|f_{p, n}-f_{p, n}^\delta\right|+
			\left|\dfrac{\alpha_{p, n}}{\beta_{p, n}}\right|\left|f_{s, n}-f_{s, n}^\delta\right|\right).
		\end{align}
		
	Now, we have
	\begin{align}\nonumber
		&\quad\dfrac{1}{2\pi\rho}\left\|\bmv_N-\bmv_N^\delta\right\|^2_{\left(L^2(\Gamma_\rho)\right)^2}\nonumber\\
		&=\sum_{n=-N}^N\left(\left|\alpha_{p,n}\left(\hat{\phi}_{p,n}-\hat{\phi}_{p,n}^\delta\right)
		+\dfrac{\mathrm{i}n}{\rho}\beta_{s,n}\left(\hat{\phi}_{s,n}-\hat{\phi}_{s,n}^\delta\right)\right|^2\right.\nonumber\\
		&\quad\left.+\left|\dfrac{\mathrm{i}n}{\rho}\beta_{p,n}\left(\hat{\phi}_{p,n}-\hat{\phi}_{p,n}^\delta\right)
		-\alpha_{s, n}\left(\hat{\phi}_{s,n}-\hat{\phi}_{s,n}^\delta\right)\right|^2\nonumber\right)\\
		&\le I_{p,0}+I_{s,0}+2\sum_{\xi\in\{p,s\}}\sum_{1\le|n|\le N}\left(|\alpha_{\xi,n}|^2+\dfrac{n^2}{\rho^2}|\beta_{\xi,n}|^2\right)
		\left|\hat{\phi}_{\xi,n}-\hat{\phi}_{\xi,n}^\delta\right|^2\nonumber\\
		&=I_{p,0}+I_{s,0}+2\left(\sum_{1\le|n|\le4k_s\rho}+\sum_{4k_s\rho+1\le|n|\le N}\right)\left(I_{p,n}+I_{s,n}\right),
		\label{eq:v_noise_expand}
	\end{align}
	where 
	\begin{align}\label{eq:I1n}
		I_{\xi, n}&:=\left(|\alpha_{\xi,n}|^2+\dfrac{n^2}{\rho^2}|\beta_{\xi,n}|^2\right)
		\left|\hat{\phi}_{\xi, n}-\hat{\phi}_{\xi, n}^\delta\right|^2, \quad\xi\in\{p, s\}.
	\end{align}
	
	In the following, we shall estimate \eqref{eq:I1n} in terms of the range of $n$. First, for $n=0,$  we have
		\begin{align*}
			\left|\hat{\phi}_{p,0}-\hat{\phi}_{p,0}^\delta\right| = \dfrac{\left|\alpha_{s,0}\right|}{\left|\Lambda_0\beta_{p,0}\beta_{s,0}\right|}\left| f_{p,0}-f_{p,0}^\delta\right|
			=\dfrac{\left| f_{p,0}-f_{p,0}^\delta\right|}{|\alpha_{p,0}|},\\ 
			\left|\hat{\phi}_{s,0}-\hat{\phi}_{s,0}^\delta\right| = \dfrac{\left|\alpha_{p,0}\right|}{\left|\Lambda_0\beta_{p,0}\beta_{s,0}\right|}\left| f_{s,0}-f_{s,0}^\delta\right|
			=\dfrac{\left| f_{s,0}-f_{s,0}^\delta\right|}{|\alpha_{s,0}|}.					
		\end{align*}
		Further, noticing \eqref{eq:I1n},
		we derive that for $\xi\in\{p,s\},$
		\begin{align}\label{eq:I1_1}
			I_{\xi,0}=\left| f_{\xi,0}-f_{\xi,0}^\delta\right|^2\le\dfrac{\delta^2}{2\pi\rho}\|\bmv\|^2_{\left(L^2(\Gamma_\rho)\right)^2}.
		\end{align}		
		
		We next consider the case for $1\le|n|\le N.$ Since $H_n^{(1)}(k_\xi\rho)\ne 0, \xi\in\{p, s\}$, there exists some constant $c_3>0$ such that
		\begin{align}\label{eq:c3}
			\left|{\beta_{\xi, n}}\right|^{-1}\le c_3,\quad  \xi\in\{p, s\}.
		\end{align}
		From \eqref{eq:Hn1_Deriv_property}, we know that for $1\le |n|\le 4 k_s\rho,$
		\[
		\left|\dfrac{\alpha_{s, n}}{\beta_{s, n}}\right|\le k_s+\dfrac{|n|}{\rho}\le 5k_s,
		\]
		then we obtain that in \eqref{eq:phi_error},
		\begin{align*}
			\left|\dfrac{\alpha_{s, n}}{\beta_{p, n}\beta_{s, n}}\right|\le 5 c_3 k_s,
		\end{align*}
		where \eqref{eq:c3} has been taken into account.
		Thus, we derive the following estimate for \eqref{eq:phi_error}--\eqref{eq:psi_error}:
		\begin{align}\nonumber
			\left|\hat{\phi}_{p, n}-\hat{\phi}_{p, n}^\delta\right| & \le \frac{5 c_3 k_s}{|\Lambda_n|}\left|f_{p, n}-f_{p, n}^\delta\right|
			+\dfrac{c_3 |n|}{\rho|\Lambda_n|}\left|f_{s, n}-f_{s, n}^\delta\right|\\
			&\le\nonumber \left(5k_s+\dfrac{|n|}{\rho}\right)\dfrac{c_3\delta}{|\Lambda_n|\sqrt{2\pi\rho}}\|\bmv\|_{\left(L^2(\Gamma_\rho)\right)^2}\\
			&\le \dfrac{9 k_s c_3\delta}{|\Lambda_n|\sqrt{2\pi\rho}}\|\bmv\|_{\left(L^2(\Gamma_\rho)\right)^2}.\label{eq:phierror}
		\end{align}
		Similarly,
		\begin{align}\label{eq:psierror}
			\left|\hat{\phi}_{s, n}-\hat{\phi}_{s, n}^\delta\right|\le\dfrac{9 k_s c_3\delta}{|\Lambda_n|\sqrt{2\pi\rho}}\|\bmv\|_{\left(L^2(\Gamma_\rho)\right)^2},
		\end{align}
		where \eqref{eq:rhs_estimate1} and \eqref{eq:rhs_estimate2} have been taken into consideration. By letting 
		$$
		M=\frac{9 k_s c_3}{\min\limits_{|n|\le 4k_s\rho}|\Lambda_n|},
		$$
		we derive from \eqref{eq:phierror} and \eqref{eq:psierror} that for $1\le|n|\le4k_s\rho,$
		\begin{align}\label{eq:phipsierr}
			\left|\hat{\phi}_{\xi, n}-\hat{\phi}_{\xi, n}^\delta\right|\le\dfrac{M\delta}{\sqrt{2\pi\rho}}\|\bm{v}\|_{\left(L^2(\Gamma_\rho)\right)^2},\quad\xi\in\{p, s\}.
		\end{align}
		
		We now estimate $I_{p, n}$ and $I_{s, n}$ for $1\le|n|\le 4k_s\rho.$
		Noticing \Cref{lem:IP15_lem3.3}, we derive that $\left|\beta_{\xi, n}\right|\le 1,\ \ \xi\in\{p, s\}$,
		which further gives that for $1\le|n|\le 4 k_s\rho,$
		\begin{align*}
			\left|\alpha_{\xi, n}\right|\le\left(k_\xi+\dfrac{|n|}{\rho}\right)\left|\beta_{\xi,n}\right|\le 5k_s,\quad   
			\dfrac{n^2}{\rho^2}|\beta_{\xi, n}|^2\le 16k_s^2.
		\end{align*}
		This together with \eqref{eq:I1n} and \eqref{eq:phipsierr}  gives
		\begin{align}\label{eq:I_2}
			I_{\xi, n}\le 41 k_s^2\dfrac{M^2\delta^2}{2\pi\rho}\|\bm{v}\|_{\left(L^2(\Gamma_\rho)\right)^2}^2,\ \ \xi\in\{p, s\}.
		\end{align}
		
		Next, we consider the case $4k_s\rho+1\le|n|\le N$. From \eqref{eq:Hn1_Deriv_property}, we deduce that
		\begin{equation}\label{eq:ab}
			\left|\alpha_{\xi, n}\right|\le\left(k_\xi+\dfrac{|n|}{\rho}\right)\left|\beta_{\xi, n}\right|\le\dfrac{5|n|}{4\rho}\left|\beta_{\xi, n}\right|,\quad \xi\in\{p, s\}.
		\end{equation}
		Noticing \eqref{eq:phi_error} and \eqref{eq:psi_error}, we further claim that
		\begin{align*}
			\left|\hat{\phi}_{p, n}-\hat{\phi}_{p, n}^\delta\right|
			&\le\frac{1}{|\Lambda_n\beta_{p, n}|}\left(\dfrac{5|n|}{4\rho}
			\left|f_{p, n}-f_{p, n}^\delta\right|+\dfrac{|n|}{\rho}
			\left|f_{s, n}-f_{s, n}^\delta\right|\right) \\
			&\le\dfrac{5|n|\delta\|\bm{v}\|_{(L^2(\Gamma_\rho))^2}}{2\rho|\Lambda_n\beta_{p, n}|\sqrt{2\pi\rho}},
		\end{align*}
		and 
		\begin{align}\label{eq:psi_2}
			\left|\hat{\phi}_{s, n}-\hat{\phi}_{s, n}^\delta\right|\le\dfrac{5|n|\delta\|\bm{v}\|_{(L^2(\Gamma_\rho))^2}}{2\rho|\Lambda_n\beta_{s, n}|\sqrt{2\pi\rho}}.
		\end{align}
		
Combining \eqref{eq:ab}--\eqref{eq:psi_2}, we estimate $I_{\xi, n}$ for $\xi\in\{p, s\}$ as follows:
		\begin{align}\nonumber
			I_{\xi, n}&\le\left(\frac{25n^2}{16\rho^2}\left|\beta_{\xi, n}\right|^2+\dfrac{n^2}{\rho^2}\left|\beta_{\xi, n}\right|^2\right)\dfrac{25n^2\delta^2}{4\rho^2|\Lambda_n\beta_{\xi, n}|^2 2\pi\rho}\|\bm{v}\|^2_{\left(L^2(\Gamma_\rho)\right)^2}\\\label{eq:I1_3}
			&\le\dfrac{25n^4\delta^2}{2\pi\rho^5|\Lambda_n|^2}\|\bmv\|^2_{\left(L^2(\Gamma_\rho)\right)^2},
		\end{align}
		with $4k_s\rho+1\le|n|\le N.$
	Collecting \eqref{eq:v_noise_expand}, \eqref{eq:I1_1},  \eqref{eq:I_2}, and \eqref{eq:I1_3}, we proceed the error estimates by pointing out that
	\begin{align*}
		\left\|\bm{v}_N-\bm{v}_N^\delta\right\|^2_{\left(L^2(\Gamma_\rho)\right)^2}
		&\le{2\delta^2}\|\bmv\|_{\left(L^2(\Gamma_\rho)\right)^2}+
		164{(8k_s\rho+1)k_s^2M^2\delta^2}\|\bmv\|^2_{\left(L^2(\Gamma_\rho)\right)^2}\\
		&\quad+\dfrac{100N^4(2N+1)\delta^2}{\rho|\Lambda_n|^2}\|\bmv\|^2_{\left(L^2(\Gamma_\rho)\right)^2}\\
		& \le C_3\delta^2+C_4 N^5\delta^2,
	\end{align*}
	where 
	\begin{align*}
		C_3 & =2\left(82(8k_s\rho+1)k_s^2M^2+1\right)\|\bm{v}\|_{(L^2(\Gamma_\rho))^2},\\
		C_4 & =\frac{300}{\rho}\|\bm{v}\|^2_{(L^2(\Gamma_\rho))^2}\max_{1\le |n|\le N}\frac{1}{|\Lambda_n|^2}.
	\end{align*}
	This completes the proof.
\end{proof}


\section{Newton type method and convergence}\label{sec:convergence}
In this section, we shall propose a novel Newton-type method to deal with the inverse elastic obstacle problem and then carry out a corresponding convergence analysis. The determination of the truncation $N$ will be also given in this section.  Before introducing the Newton-type iterative method, we shall first make the following assumptions: 

\begin{assumption}\label{assumption1}
	The boundary curve is approximated by a star-shaped curve with the shift given in the form of $h(\hat{x})=h\hat{x}.$
\end{assumption}

\begin{assumption}\label{assumption2}
	The radial derivatives do not vanish in a small closed neighborhood $U$ of $\partial D,$ namely, there exists some $\varepsilon>0$ such that
	\begin{align}\label{eq:lowboundepsilon}
		\left|\left.\dfrac{\partial\bm{u}}{\partial\hat{x}}\right|_{\Gamma}\right|\ge 4\varepsilon>0, \quad \forall\Gamma\in U.
	\end{align}
\end{assumption}

\begin{assumption}\label{assumption3}
	There exists some $\gamma>0$ such that for all $x\in\Gamma$ with $r=|x|,$ it holds that $r\ge R+\gamma.$
\end{assumption}

\begin{remark}
	We would like to point out that \Cref{assumption1} is proposed from the point of numerical implementation, whereas it is not a restriction of the target. In other words, the exact boundary curve can be either star-shaped or not. 
\end{remark}

\subsection{Newton type method}

Based on the Fourier-Bessel approximation in \Cref{sec:FB}, we shall propose an iterative scheme to deal with \Cref{problem}. Explicit computation of the derivatives will be deduced analytically in the Newton-type method.

In \eqref{eq:us_a_noisy}, we derive the approximate scattered field where the measured noisy scattered field is used. Once the approximated scattered field $\bm{v}_N^\delta=(\bm{v}_{N,1}^\delta,\bm{v}_{N,2}^\delta)^\top$ is available, the gradient of $\bm{v}_N^\delta$ under the polar coordinate $(r,\theta):x=r(\cos\theta,\sin\theta)$ can be calculated by
\begin{align*}
	\nabla\bm{v}_N^\delta & =\left(\partial_{x_1}\bm{v}_{N}^\delta,\partial_{x_2}\bm{v}_{N}^\delta\right), \\
	& =\left(\partial_r\bm{v}_{N}^\delta\cos\theta-\dfrac{1}{r}\partial_\theta \bm{v}_{N}^\delta\sin\theta, \partial_r\bm{v}_{N}^\delta\sin\theta+\dfrac{1}{r}\partial_\theta \bm{v}_{N}^\delta\cos\theta\right), 
\end{align*}
where
\begin{align*}
	\partial_r\bm{v}_{N}^\delta  & = \sum_{n=-N}^{N}\left(\kappa_{p, n}(r)\hat{\phi}_{p, n}^\delta
	+\dfrac{\mathrm{i}n}{r}{\alpha_{s, n}(r)}\hat{\phi}_{s, n}^\delta
	-\dfrac{\mathrm{i}n}{r^2}{\beta_{s, n}(r)}\hat{\phi}_{s, n}^\delta\right)\bm{U}_n\\ 
	\nonumber &\quad+
	\left(\dfrac{\mathrm{i}n}{r}{\alpha_{p, n}(r)}\hat{\phi}_{p, n}^\delta-\dfrac{\mathrm{i}n}{r^2}\beta_{p, n}(r)\hat{\phi}_{p, n}^\delta
	-\kappa_{s, n}(r)\hat{\phi}_{s, n}^\delta\right)\bm{V}_n,\\ 
	\partial_{\theta}\bm{v}_{N}^\delta  & = \sum_{n=-N}^N\left(\alpha_{p, n}(r)\hat{\phi}_{p, n}^\delta
	+\dfrac{\mathrm{i}n}{r}\beta_{s, n}(r)\hat{\phi}_{s, n}^\delta\right)\left(\mathrm{i}n
	\boldsymbol{U}_n+\bm{V}_n\right) \\
	&\quad+\left(\dfrac{\mathrm{i}n}{r}\beta_{p, n}(r)\hat{\phi}_{p, n}^\delta-\alpha_{s, n}(r)
	\hat{\phi}_{s, n}^\delta\right)\left(\mathrm{i}n\boldsymbol{V}_n-\boldsymbol{U}_n\right),
\end{align*}
and 
$$
\kappa_{\xi, n}(r)={\alpha'_{\xi, n}}(r)=k_\xi^2\dfrac{{H_n^{(1)}}''(k_\xi r)}{H_n^{(1)}(k_\xi R)},\quad \xi\in\{p, s\}.
$$

Once the approximate scattered field $\bm{v}_N^\delta$ together with its gradient has been derived, we are now to propose the novel Newton method. For this purpose, we define the operator $G_N$ by
\[
G_N:\Gamma\mapsto\boldsymbol{u}_N^\delta,
\]
which maps the boundary contour $\Gamma$ to the approximate total field $\bm{u}_N^\delta=\bm{u}^i+\bm{v}_N^\delta.$ Let $p$ be the parameterization of the boundary contour $\Gamma.$ To determine the position of $\partial D$ where the Dirichlet boundary condition is fulfilled, we shall find $p$ such that
\begin{align}\label{eq:GN}
	G_N(p)={\bm 0}.
\end{align}

In the spirit of the Newton method, we now linearize \eqref{eq:GN} and update the parameterization $p$ through
\begin{align}\label{eq:Newton}
	\left\{
	\begin{aligned}
		&G_N(p_m)+G'_N(p_m)h_m={\bm 0},\\
		& p_{m+1}=p_m+h_m,
	\end{aligned}
	\right.\quad
	m=0,1,\cdots.
\end{align}

Once the initial guess $\Gamma_0$ is selected, we can obtain the approximation $\Gamma_m$ at the $m$-th step through \eqref{eq:Newton}, where the derivative $G'_N(p)$ can be explicitly calculated by
\[
\Big[G'_N(p)\Big]=\nabla(\boldsymbol{u}^i+\bm{v}_N^\delta).
\]

\subsection{Convergence Analysis}

In this subsection, we shall analyze the convergence of the Newton-type method proposed in the last subsection. Under \Cref{assumption1} -- \Cref{assumption3}, we first establish the error estimates on the iteration curve $\Gamma_m,\,m=0,1,\cdots.$ Afterwards, we shall address the convergence results.
\begin{lemma}
	For $n\in\mathbb{Z}$ and $t\in\mathbb{R}\backslash\{0\}$ such that $H_n^{(1)}(t)\ne0,$ it holds that
	\begin{align*}
		\dfrac{{H_n^{(1)}}''(t)}{H_n^{(1)}(t)}\le\dfrac{n(n+1)}{t^2}+\dfrac{1}{t}.
	\end{align*}
\end{lemma}
\begin{proof}
	From the fact that $t{H_n^{(1)}}'(t)=tH_{n-1}^{(1)}(t)-nH_n^{(1)}(t)$ and the recursive formula $2nH_n^{(1)}(t)=t\left(H_{n+1}^{(1)}(t)+H_{n-1}^{(1)}(t)\right),$ we derive that  
	\begin{align*}
		{H_n^{(1)}}''(t)&=-{H_{n+1}^{(1)}}'(t)-\dfrac{n}{t^2}H_n^{(1)}(t)+\dfrac{n}{t}\left(-H_{n+1}^{(1)}(t)+\dfrac{n}{t}H_n^{(1)}(t)\right)\\
		&=\dfrac{n^2-n-t^2}{t^2}H_n^{(1)}(t)+\dfrac{1}{t}H_{n+1}^{(1)}(t).
	\end{align*}
	Hence,
	\begin{align*}
		\dfrac{{H_n^{(1)}}''(t)}{H_n^{(1)}(t)}&=\dfrac{n^2-n-t^2}{t^2}+
		\dfrac{1}{t}\dfrac{H_{n+1}^{(1)}(t)}{H_n^{(1)}(t)}\\
		&=\dfrac{n(n-1)}{t^2}-1+\dfrac{2n}{t^2}-\dfrac{1}{t}\dfrac{H_{n-1}^{(1)}(t)}{H_n^{(1)}(t)}\\
		&\le\dfrac{n(n+1)}{t^2}+\dfrac{1}{t}.
	\end{align*}
\end{proof}
From the above, it is easy to check that 
\begin{align}\label{eq:alpha_est}
	\left|\alpha_{\xi,n}(r)\right|\le\left(\dfrac{|n|}{r}+k_\xi\right)\left|\beta_{\xi,n}(r)\right|,\\
	\left|\kappa_{\xi,n}(r)\right|\le\dfrac{2n^2+k_\xi r}{r^2}\left|\beta_{\xi,n}(r)\right|.\label{eq:kappa_est}
\end{align}

\begin{theorem}
	Under assumption \eqref{eq:assume},
	the following error estimate on each iterative curve $\Gamma_m,\,m=1,2,\cdots,$ holds:
	\begin{equation*}
		\left\|\bm{v}-\bm{v}_N\right\|_{\left(L^2(\Gamma_m)\right)^2}\le C_5N\tau_2^{2-N},
	\end{equation*}
	where $\tau_2=({R+\gamma})/{R}$, and $C_5$ is a constant independent of $N$.
\end{theorem}

\begin{proof}
	From \eqref{eq:us} and \eqref{eq:us_a}, it can be seen that
	\begin{align}\nonumber
		&\quad\left\|\bm{v}-\bm{v}_N\right\|^2_{\left(L^2(\Gamma_m)\right)^2}\\
		&=\int_{\Gamma_m}\sum_{|n|>N}\nonumber
		\left|\alpha_{p, n}(r)\hat{\phi}_{p, n}+\dfrac{\mathrm{i}n}{r}\beta_{s, n}(r)\hat{\phi}_{s, n}\right|^2
		+\left|\dfrac{\mathrm{i}n}{r}\beta_{p, n}(r)\hat{\phi}_{p, n}-\alpha_{s, n}(r)\hat{\phi}_{s, n}\right|^2\mathrm{d}s(x)\\
		\nonumber & \le 2\sum_{\xi\in\{p, s\}}\sum_{|n|>N}\int_{\Gamma_m}\left(\left|\alpha_{\xi, n}(r)\right|^2+\dfrac{n^2}{r^2}\left|\beta_{\xi, n}(r)\right|^2\right)\left|\hat{\phi}_{\xi, n}\right|^2\mathrm{d}s(x)\\
		\nonumber & \le 2|\Gamma_m|\max_{x\in\Gamma_m}\sum_{\xi\in\{p, s\}}\sum_{|n|>N}\left(\left|\alpha_{\xi, n}(r)\right|^2+\dfrac{n^2}{r^2}\left|\beta_{\xi, n}(r)\right|^2\right)\left|\hat{\phi}_{\xi, n}\right|^2\\
		& \le 2|\Gamma_m|\sum_{\xi\in\{p, s\}}\sum_{|n|>N}\left(\max_{x\in\Gamma_m}\left|\alpha_{\xi, n}(r)\right|^2+\dfrac{n^2}{(R+\gamma)^2}\left|\beta_{\xi, n}(R+\gamma)\right|^2\right)\left|\hat{\phi}_{\xi, n}\right|^2\label{eq:err_gamma}
	\end{align}
	In \eqref{eq:err_gamma}, for $\xi\in\{p, s\},$
	\begin{align*}
		\left|\alpha_{\xi, n}(r)\right|&\le\left(k_\xi+\dfrac{|n|}{r}\right)\left|\beta_{\xi, n}(r)\right|
		\le\left(k_\xi+\dfrac{|n|}{R+\gamma}\right)\left|\beta_{\xi, n}(R+\gamma)\right|\\
		&\le 2\left(k_\xi+\dfrac{|n|}{R+\gamma}\right)\mathrm{e}^{k_\xi(R+\gamma)}\tau_2^{-|n|},
	\end{align*}
	where $\tau_2=({R+\gamma})/{R},$ and \Cref{lem:IP15_lem3.3,lem:IPI} have been taken into account. Noticing 
	$$
	|n|>N>\dfrac{1}{2}(\mathrm{e}k_s\rho+1),
	$$ 
	we derive that
	\[
	\dfrac{|n|}{R+\gamma}\ge\dfrac{1}{2}\left(\mathrm{e}k_s+\dfrac{1}{R+\gamma}\right)>k_\xi,\quad\xi\in\{p,s\},
	\]
	which further gives that
	\begin{align}\label{eq:alpha_estimate}
		\left|\alpha_{\xi, n}(r)\right|&
		\le\dfrac{4|n|}{R+\gamma}\mathrm{e}^{k_\xi(R+\gamma)}\tau_2^{-|n|}.
	\end{align}
	Meanwhile, noticing \eqref{eq:H_n1_estimate}, it holds that
	\begin{align}\label{eq:beta_estimate2}
		\left|\beta_{\xi, n}(r)\right|
		\le{2}\mathrm{e}^{k_\xi(R+\gamma)}\tau_2^{-|n|}.
	\end{align}
	
	By \eqref{eq:alpha_estimate} and Cauchy inequality, we have
	\begin{align}
		\sum_{|n|>N}\left|\alpha_{\xi, n}(r)\right|^2\left|\hat{\phi}_{\xi, n}\right|^2
		& \le\left(\sum_{|n|>N}\left|\alpha_{\xi, n}(r)\right|^2\right)
		\left(\sum_{|n|>N}\left|\hat{\phi}_{\xi, n}\right|^2\right)\notag \\
		& \le\dfrac{8\mathrm{e}^{2k_\xi(R+\gamma)}}{\pi R(R+\gamma)^2}\left\|{\phi}_\xi\right\|^2_{L^2(\Gamma_R)}
		\sum_{|n|>N}n^2\tau_2^{-2|n|},\quad \xi\in\{p, s\}. \label{eq:alpha_phi}
	\end{align}
	Similarly, for $\xi\in\{p, s\}$, one can obtain
	\begin{equation}\label{eq:beta_psi_n}
		\sum_{|n|> N}{n^2}\left|\beta_{\xi, n}(r)\right|^2\left|\hat{\phi}_{\xi, n}\right|^2
		\le \frac{2\mathrm{e}^{2k_\xi(R+\gamma)}}{\pi R}\left\|{\phi}_\xi\right\|^2_{L^2(\Gamma_R)}
		\sum_{|n|>N}n^2\tau_2^{-2|n|}.
	\end{equation}
	
	Combining \eqref{eq:err_gamma}, \eqref{eq:alpha_phi}, \eqref{eq:beta_psi_n},  and a result similar to that in the appendix, we derive that 
	\begin{align*}
		\left\|\bm{v}-\bm{v}_N\right\|^2_{(L^2(\Gamma_m))^2} & \le\dfrac{40|\Gamma_m|\mathrm{e}^{2k_s(R+\gamma)}}{\pi R(R+\gamma)^2}\left(\|\phi_p\|^2_{L^2(\Gamma_R)}+\|\phi_s\|^2_{L^2(\Gamma_R)}\right)\sum_{|n|>N}n^2\tau_2^{-2|n|} \\
		& \le \tilde{C}_5 N^2\tau_2^{4-2N},
	\end{align*}
	where 
	$$
	\tilde{C}_5=\dfrac{320|\Gamma_m|\mathrm{e}^{2k_s(R+\gamma)}}{\pi R(R+\gamma)^2(\tau_2^2-1)^3}\left(\|\phi_p\|^2_{L^2(\Gamma_R)}+\|\phi_s\|^2_{L^2(\Gamma_R)}\right).
	$$
	
	This completes the proof.
\end{proof}
	
\begin{theorem}
	Under assumption \eqref{eq:assume}, there exist some positive constants $C_6$ and $C_7$ independent of $N$ and $\delta$ such that 
		\begin{equation*}
			\|\bm{v}_N-\bm{v}_N^\delta\|_{(L^2(\Gamma_m))^2}^2\le C_6\delta^2+C_7N^5\tau_1^{2N}\delta^2.
		\end{equation*}
\end{theorem}

\begin{proof}
	 From \eqref{eq:us_a} and \eqref{eq:us_a_noisy}, we obtain that
	\begin{align*}
		& \quad\|\bm{v}_N-\bm{v}_N^\delta\|_{\left(L^2(\Gamma_m)\right)^2}^2 \\
		& =\sum_{|n|\le N}\int_{\Gamma_m}\left|\alpha_{p, n}(r)\left(\hat{\phi}_{p, n}-\hat{\phi}_{p, n}^\delta\right)+\dfrac{\ii n}{r}\beta_{s, n}(r)\left(\hat{\phi}_{s, n}-\hat{\phi}_{s, n}^\delta\right)\right|^2 \\
		&\quad +\left|\dfrac{\ii n}{r}\beta_{p, n}(r)\left(\hat{\phi}_{p, n}-\hat{\phi}_{p, n}^\delta\right)-\alpha_{s, n}(r)\left(\hat{\phi}_{s, n}-\hat{\phi}_{s, n}^\delta\right)\right|^2 \mathrm{d}s(x)\\
		& \le 2\sum_{\xi\in\{p, s\}}\sum_{|n|\le N}\int_{\Gamma_m}\left(\left|\alpha_{\xi, n}(r)\right|^2+\dfrac{n^2}{r^2}\left|\beta_{\xi, n}(r)\right|^2\right)\left|\hat{\phi}_{\xi, n}-\hat{\phi}_{\xi, n}^\delta\right|^2\mathrm{d}s(x).
	\end{align*}
	Noticing \eqref{eq:Hn1_property}, we obtain that for $\xi\in\{p, s\},$
	\begin{align*}
		\left|\beta_{\xi, n}(r)\right|&\le\left|\beta_{\xi, n}(R+\gamma)\right|,\\
		\left|\alpha_{\xi, n}(r)\right|&\le\left(k_\xi +\dfrac{|n|}{r}\right)\left|\beta_{\xi, n}(r)\right|
		\le\left(k_\xi +\dfrac{|n|}{R+\gamma}\right)\left|\beta_{\xi, n}(R+\gamma)\right|.
	\end{align*}
	Thus,
	\begin{align*}
		&\quad\|\bm{v}_N-\bm{v}_N^\delta\|_{\left(L^2(\Gamma_m)\right)^2}^2 \\
		& \le 6\sum_{\xi\in\{p, s\}}\sum_{|n|\le N}\int_{\Gamma_m}\left(k_\xi^2+\dfrac{n^2}{(R+\gamma)^2}\right)\left|\beta_{\xi, n}(R+\gamma)\right|^2\left|\hat{\phi}_{\xi, n}-\hat{\phi}_{\xi, n}^\delta\right|^2\mathrm{d}s(x)\\
		& = 6|\Gamma_m|\sum_{\xi\in\{p, s\}}\sum_{|n|\le N}\left(k_\xi^2+\dfrac{n^2}{(R+\gamma)^2}\right)\left|\beta_{\xi, n}(R+\gamma)\right|^2\left|\hat{\phi}_{\xi, n}-\hat{\phi}_{\xi, n}^\delta\right|^2
	\end{align*}
	
	According to \Cref{lem:IP15_lem3.3},  we know that $\left|\beta_{\xi, n}(R+\gamma)\right|\le 1, \xi\in\{p, s\}$. Therefore,
	\begin{align*}
		&\quad \|\bm{v}_N-\bm{v}_N^\delta\|_{\left(L^2(\Gamma_m)\right)^2}^2\\
		&\le 6|\Gamma_m|\sum_{|n|\le N}\left(k_s^2+\dfrac{n^2}{(R+\gamma)^2}\right)\sum_{\xi\in\{p, s\}}\left|\hat{\phi}_{\xi, n}-\hat{\phi}_{\xi, n}^\delta\right|^2\nonumber \\
		& \le 6|\Gamma_m|\left(\sum_{|n|\le 4 k_s\rho}+\sum_{4k_s\rho+1\le|n|\le N}\right)\left(k_s^2+\dfrac{n^2}{(R+\gamma)^2}\right)\sum_{\xi\in\{p, s\}}\left|\hat{\phi}_{\xi, n}-\hat{\phi}_{\xi, n}^\delta\right|^2\nonumber \\
		&=  6|\Gamma_m|(I_1+I_2),\nonumber
	\end{align*}
	where 
	\begin{align*}
		I_1& =\sum_{|n|\le 4 k_s\rho}\left(k_s^2+\dfrac{n^2}{(R+\gamma)^2}\right)\sum_{\xi\in\{p, s\}}\left|\hat{\phi}_{\xi, n}-\hat{\phi}_{\xi, n}^\delta\right|^2,\\
		I_2 & =\sum_{4k_s\rho+1\le|n|\le N}\left(k_s^2+\dfrac{n^2}{(R+\gamma)^2}\right)\sum_{\xi\in\{p, s\}}\left|\hat{\phi}_{\xi, n}-\hat{\phi}_{\xi, n}^\delta\right|^2.
	\end{align*}
	
	To estimate $I_1$ and $I_2,$ we first point out that, similar to the analysis in the proof of \Cref{thm:error1}, we know that there exists some constant $M_2>0$ independent of $k_\xi,\,\xi\in\{p, s\}$ and $\delta$ such that 
	\begin{equation}\label{eq:phierr}
		\left|\hat{\phi}_{\xi, n}-\hat{\phi}_{\xi, n}^\delta\right|\le\dfrac{M_2\delta}{\sqrt{2\pi\rho}}\|\bm{v}\|_{\left(L^2(\Gamma_\rho)\right)^2},\quad |n|\le 4 k_s\rho,
	\end{equation}
 and 
	\begin{equation}\label{eq:psierr}
		\left|\hat{\phi}_{\xi, n}-\hat{\phi}_{\xi, n}^\delta\right|\le\dfrac{5|n|}{2\rho\left|\Lambda_n\beta_{\xi, n}(\rho)\right|}\dfrac{\delta}{\sqrt{2\pi\rho}}\|\bm{v}\|_{\left(L^2(\Gamma_\rho)\right)^2},\quad 4k_s\rho+1\le|n|\le N.	
	\end{equation}
	
	To estimate $I_1,$ we notice that
	$$
	    k_s^2+\frac{n^2}{(R+\gamma)^2}\le k_s^2+\frac{16k_s^2\rho^2}{(R+\gamma)^2}\le\frac{17k_s^2\rho^2}{(R+\gamma)^2},
	$$
	which illustrates that
	\begin{align*}
		I_1&\le\sum_{|n|\le 4 k_s\rho}\frac{17k_s^2\rho^2}{(R+\gamma)^2}\left(\left|\hat{\phi}_{p, n}-\hat{\phi}_{p, n}^\delta\right|^2+\left|\hat{\phi}_{s, n}-\hat{\phi}_{s, n}^\delta\right|^2\right)\\
		\nonumber &\le\frac{17(8k_s\rho+1)k_s^2\rho}{\pi(R+\gamma)^2}{M_2^2\delta^2}\|\bmv\|_{\left(L^2(\Gamma_\rho)\right)^2}^2\\
		& = \tilde{C}_6\delta^2,
	\end{align*}
	with
	$$
		\tilde{C}_6=\frac{17(8k_s\rho+1)k_s^2\rho}{\pi(R+\gamma)^2}{M_2^2}\|\bm{v}\|_{\left(L^2(\Gamma_\rho)\right)^2}^2.
	$$
	
	Now, we consider the estimate for $I_2.$ For later analysis, we first recall that (\cite{LiWangWangZhao16}) $\Lambda_n\ne 0,\ \forall n\in\mathbb{Z},$ thus there exists some constant $c_3>0$ such that
	\begin{equation}\label{eq:c3_new}
	\frac{1}{|\Lambda_n|^2}\le c_3.
	\end{equation}

	From \Cref{lem:IPI}, we know that for $\xi\in\{p, s\},$
	\begin{equation*}
		\left|\beta_{\xi, n}(\rho)\right|^{-2}=\frac{\left|H_n^{(1)}(k_\xi R)\right|^2}{\left|H_n^{(1)}(k_\xi\rho)\right|^2}
		\le 4\ee^{2k_\xi R}\tau_1^{2n}\le 4\ee^{2k_s R}\tau_1^{2n}.
	\end{equation*}
This, together with \eqref{eq:psierr} and \eqref{eq:c3_new}, gives
	\begin{align}\nonumber
		\left|\hat{\phi}_{\xi, n}-\hat{\phi}_{\xi, n}^\delta\right|^2
		&\le\frac{25n^2}{4\rho^2|\Lambda_n|^2}4\ee^{2k_sR}\tau_1^{2n}\frac{\delta^2}{2\pi\rho}\|\bmv\|^2_{(L^2(\Gamma_\rho))^2}\\
		\label{eq:phierrorpsi} &\le\frac{25c_3\ee^{2k_sR}}{2\pi\rho^3}\|\bmv\|^2_{(L^2(\Gamma_\rho))^2}n^2\delta^2\tau_1^{2n},\quad\text{for }\xi\in\{p, s\}.
	\end{align}
	Further, we proceed with the proof by pointing out that 
	\begin{align*}
		I_2&\le\sum_{4k_s\rho+1\le|n|\le N}\frac{2n^2}{(R+\gamma)^2}\sum_{\xi\in\{p, s\}}\left|\hat{\phi}_{\xi, n}-\hat{\phi}^\delta_{\xi, n}\right|^2\\
		& \le\frac{50c_3\ee^{2k_sR}\delta^2\|\bmv\|^2_{(L^2(\Gamma_\rho))^2}}{\pi\rho^3(R+\gamma)^2}\sum_{4k_s\rho+1\le|n|\le N}n^4\tau_1^{2n}\\
		&\le\frac{50c_3\ee^{2k_sR}\delta^2\|\bmv\|^2_{(L^2(\Gamma_\rho))^2}}{\pi\rho^3(R+\gamma)^2}(2N+1)N^4\tau_1^{2n}\\&\le\frac{150c_3\ee^{2k_sR}\|\bmv\|^2_{(L^2(\Gamma_\rho))^2}}{\pi\rho^3(R+\gamma)^2}N^5\delta^2\tau_1^{2N} \\
		& = \tilde{C}_7 N^5\tau_1^{2N}\delta^2.
	\end{align*}
	where 
	$$
	\tilde{C}_7=\frac{150c_3\ee^{2k_sR}\|\bmv\|^2_{(L^2(\Gamma_\rho))^2}}{\pi\rho^3(R+\gamma)^2}.
	$$

	Finally, let $C_6=6\tilde{C}_6|\Gamma_m|,\,C_7=6\tilde{C}_7|\Gamma_m|$, then we obtain the estimate
	\begin{equation*}
		\|\bm{v}_{N}-\bm{v}_N^\delta\|^2_{(L^2(\Gamma_m))^2}\le C_6\delta^2+C_7N^5\tau_1^{2N}\delta^2.
	\end{equation*}
	\end{proof}
	
	\begin{theorem}
	Under \Cref{assumption1} and assumption \eqref{eq:assume}, it holds that
	\begin{equation}\label{eq:deriverror}
		\left\|\dfrac{\partial\bmv}{\partial\hat{x}}-\dfrac{\partial\bm{v}_N}{\partial\hat{x}}\right\|_{\left(L^2(\Gamma_m)\right)^2}^2 \le C_8N^4\tau_2^{-2N},
	\end{equation}
	where $\tau_2=(R+\gamma)/R$, and $C_8$ is a constant independent of $N.$
	\end{theorem}
	
	\begin{proof}
	Under \Cref{assumption1}, we derive that
	\begin{align*}
		\left|\dfrac{\partial\bmv}{\partial\hat{x}}-\dfrac{\partial\bm{v}_N}
		{\partial\hat{x}}\right|=\left|{\partial_r(\bmv-\bm{v}_N)}\right|,\quad \text{on}\ \Gamma_m.
	\end{align*}
	\begin{align*}
	\partial_r(\bmv-\bm{v}_N)= & \sum_{|n|>N}\left(\kappa_{p, n}(r)\hat{\phi}_{p, n}+\dfrac{\ii n}{r}\alpha_{s, n}(r)\hat{\phi}_{s, n}-\dfrac{\ii n}{r^2}\beta_{s, n}(r)\hat{\phi}_{s, n}\right)\bm{U}_n\\
	&+\left(-\dfrac{\ii n}{r^2}\beta_{p, n}(r)\hat{\phi}_{p, n}+\dfrac{\ii n}{r}\alpha_{p, n}(r)\hat{\phi}_{p, n}-\kappa_{s, n}(r)\hat{\phi}_{s, n}\right)\bm{V}_n.
	\end{align*}
	Thus,
	\begin{align*}
		&\quad \int_{\Gamma_m}\left|\dfrac{\partial\bmv}{\partial\hat{x}}-\dfrac{\partial\bm{v}_N}
		{\partial\hat{x}}\right|^2\mathrm{d}s(x)=\int_{\Gamma_m} \left|\partial_r(\bmv-\bm{v}_N)\right|^2\mathrm{d}s(x)\\\nonumber
		& =\sum_{|n|>N}\int_{\Gamma_m}\left(\left|\kappa_{p, n}(r)\hat{\phi}_{p, n}+\dfrac{\ii n}{r}\alpha_{s, n}(r)\hat{\phi}_{s, n}-\dfrac{\ii n}{r^2}\beta_{s, n}(r)\hat{\phi}_{s, n}\right|^2\right.\\
		& \quad\left.+\left|-\dfrac{\ii n}{r^2}\beta_{p, n}(r)\hat{\phi}_{p, n}+\dfrac{\ii n}{r}\alpha_{p, n}(r)\hat{\phi}_{p, n}-\kappa_{s, n}(r)\hat{\phi}_{s, n}\right|^2\right)\mathrm{d}s(x) \\
		& \le 3\sum_{|n|>N}\int_{\Gamma_m}
		\left(\left|\kappa_{p, n}(r)\right|^2+\dfrac{n^2}{r^4}\left|\beta_{p, n}(r)
		\right|^2+\dfrac{n^2}{r^2}\left|\alpha_{p, n}(r)\right|^2\right)\left|\hat{\phi}_{p, n}\right|^2 \\
		& \quad+\left(\left|\kappa_{s, n}(r)\right|^2+\dfrac{n^2}{r^4}\left|\beta_{s, n}(r)
		\right|^2+\dfrac{n^2}{r^2}\left|\alpha_{s, n}(r)\right|^2\right)\left|\hat{\phi}_{s, n}\right|^2
		\mathrm{d}s(x).
	\end{align*}
	
	In view of \eqref{eq:alpha_est}--\eqref{eq:kappa_est} and \Cref{lem:IPI}, we derive that for $x\in\Gamma_m$ with $r=|x|,$
	\begin{align*}
		\left|\beta_{\xi, n}(r)\right|&\le 2\mathrm{e}^{k_\xi(R+\gamma)}\tau_2^{-|n|},\quad
		\left|\alpha_{\xi, n}(r)\right|\le 2\left(\dfrac{|n|}{R+\gamma}+k_\xi\right)\mathrm{e}^{k_\xi(R+\gamma)}\tau_2^{-|n|},\\ 
		\left|\kappa_{\xi, n}(r)\right|&
		\le 2\left(\dfrac{2n^2+k_\xi (R+\gamma)}{(R+\gamma)^2}\right)\mathrm{e}^{k_\xi(R+\gamma)}\tau_2^{-|n|}.
	\end{align*}	
	These together with \eqref{eq:alpha_estimate}--\eqref{eq:beta_estimate2} and the assumption that $N>4k_s\rho$ lead to
	\begin{align*}
		&\quad\int_{\Gamma_m}\left|\dfrac{\partial\bm{v}}{\partial\hat{x}}-\dfrac{\partial\bm{v}_N}{\partial\hat{x}}\right|^2\mathrm{d}s(x)\\
		&\le12\mathrm{e}^{2k_s(R+\gamma)}\sum_{|n|>N}\tau_2^{-2|n|}\int_{\Gamma_m}\left(\left(\dfrac{2n^2+k_s (R+\gamma)}{(R+\gamma)^2}\right)^2+
		\dfrac{n^2}{(R+\gamma)^4}\right.\\
		&\quad\left.+\dfrac{n^2}{(R+\gamma)^2}\left(\dfrac{|n|}{R+\gamma}+k_s\right)^2\right)\left(\left|\hat{\phi}_{p, n}\right|^2+\left|\hat{\phi}_{s, n}\right|^2\right)\mathrm{d}s(x)\\
		&\le \frac{12\mathrm{e}^{2k_s(R+\gamma)}}{(R+\gamma)^4}\sum_{|n|>N}\tau_2^{-2|n|}\int_{\Gamma_m}\left({13n^4}+
		{n^2}\right)\left(\left|\hat{\phi}_{p, n}\right|^2+\left|\hat{\phi}_{s, n}\right|^2\right)\mathrm{d}s(x)\\
		&\le\dfrac{168|\Gamma_m|\mathrm{e}^{2k_s(R+\gamma)}}{(R+\gamma)^4}\sum_{|n|>N}n^4\tau_2^{-2|n|}\left(\left|\hat{\phi}_{p, n}\right|^2+\left|\hat{\phi}_{s, n}\right|^2\right),
	\end{align*}
	where we have used the fact that $n^2\le n^4$ for $|n|\ge 1$. Taking the Cauchy-Schwarz inequality into account, we proceed the proof by showing that 
	\begin{align*}
		&\quad\int_{\Gamma_m}\left|\dfrac{\partial\bm{v}}{\partial\hat{x}}-\dfrac{\partial\bm{v}_N}{\partial\hat{x}}\right|^2\mathrm{d}s(x)\\
		&\le\dfrac{168|\Gamma_m|\mathrm{e}^{2k_s(R+\gamma)}}{(R+\gamma)^4}\left(\sum_{|n|>N}n^4\tau_2^{-2|n|}\right)\sum_{|n|>N}\left(\left|\hat{\phi}_{p, n}\right|^2+\left|\hat{\phi}_{s, n}\right|^2\right)\\
		&\le\dfrac{84|\Gamma_m|\mathrm{e}^{2k_s(R+\gamma)}}{(R+\gamma)^4\pi R}\left(\sum_{|n|>N}n^4\tau_2^{-2|n|}\right)\left(\left\|\phi_p\right\|_{L^2(\Gamma_R)}^2+\left\|\phi_s\right\|_{L^2(\Gamma_R)}^2\right).
	\end{align*}
	
	Analogous to the proof in the Appendix, it can be shown that there exists a positive constant $\tilde{C}_8>0$ such that 
	\begin{equation*}
		\sum_{|n|>N}n^4\tau_2^{-2|n|}\le \tilde{C}_8N^4\tau_2^{-2N}.
	\end{equation*}
	By letting 
	$$
	C_8=\dfrac{84|\Gamma_m|\mathrm{e}^{2k_s(R+\gamma)}}{(R+\gamma)^4\pi R}\left(\left\|\phi_p\right\|_{L^2(\Gamma_R)}^2+\left\|\phi_s\right\|_{L^2(\Gamma_R)}^2\right)\tilde{C}_8,
	$$ 
	we could obtain the following estimate:
	\begin{equation*}
		\left\|\dfrac{\partial\bmv}{\partial\hat{x}}-\dfrac{\partial\bm{v}_N}{\partial\hat{x}}\right\|^2_{\left(L^2(\Gamma_m)\right)^2}
		\le C_8N^4\tau_2^{-2N},
	\end{equation*}
	which completes the proof of \eqref{eq:deriverror}.
	\end{proof}
	
	\begin{theorem}
			Under \Cref{assumption1} and assumption \eqref{eq:assume}, it holds that
			\begin{align*}
				\left\|\dfrac{\partial\bm{v}_N}{\partial\hat{x}}-\dfrac{\partial\bm{v}_N^\delta}{\partial\hat{x}}\right\|_{\left(L^2(\Gamma_m)\right)^2}^2 \le C_9\delta^2+C_{10}N^7\tau_1^{2N}\delta^2,
			\end{align*}
			where $C_9$ and $C_{10}$ are constants independent of $N$ and $\delta$.
	\end{theorem}
	
\begin{proof}
\Cref{assumption1} implies that
\begin{align*}
	\left|\dfrac{\partial\bm{v}_N}{\partial\hat{x}}-\dfrac{\partial\bm{v}_N^\delta}
	{\partial\hat{x}}\right|=\left|{\partial_r(\bm{v}_N-\bm{v}_N^\delta)}\right|.
\end{align*}
Thus,
	\begin{align*}
		&\quad\left\|\dfrac{\partial \bm{v}_N}{\partial\hat{x}}-\dfrac{\partial\bm{v}_N^\delta}{\partial\hat{x}}\right\|^2_{(L^2(\Gamma_m))^2}\\
		&=\sum_{|n|\le N}\int_{\Gamma_m}\left|\kappa_{p, n}(r)\left(\hat{\phi}_{p, n}-\hat{\phi}_{p, n}^\delta\right)+
		\dfrac{\ii n}{r}\left(\alpha_{s, n}(r)
		-\dfrac{1}{r}\beta_{s, n}(r)\right)\left(\hat{\phi}_{s, n}-\hat{\phi}_{s, n}^\delta\right)\right|^2\\
		&\quad+\left|\dfrac{\ii n}{r}\left(\alpha_{p, n}(r)
		-\dfrac{1}{r}\beta_{p, n}(r)\right)\left(\hat{\phi}_{p, n}-\hat{\phi}_{p, n}^\delta\right)-\kappa_{s, n}(r)\left(\hat{\phi}_{s, n}-\hat{\phi}_{s, n}^\delta\right)\right|^2\mathrm{d}s(x)\\
		& \le 3\sum_{|n|\le N}\int_{\Gamma_m}\left(\left|\kappa_{p, n}(r)\left(\hat{\phi}_{p, n}-\hat{\phi}_{p, n}^\delta\right)\right|^2+\left|\dfrac{\ii n}{r}\alpha_{s, n}(r)\left(\hat{\phi}_{s, n}-\hat{\phi}_{s, n}^\delta\right)\right|^2\right. \\
		&\quad+\left|\dfrac{\ii n}{r^2}\beta_{s, n}(r)\left(\hat{\phi}_{s, n}-\hat{\phi}_{s, n}^\delta\right)\right|^2
		   +\left|\dfrac{\ii n}{r^2}\beta_{p, n}(r)\left(\hat{\phi}_{p, n}-\hat{\phi}_{p, n}^\delta\right)\right|^2 \\
		 &\quad +\left.\left|\dfrac{\ii n}{r}\alpha_{p, n}(r)\left(\hat{\phi}_{p, n}-\hat{\phi}_{p, n}^\delta\right)\right|^2+\left|\kappa_{s, n}(r)\left(\hat{\phi}_{s, n}-\hat{\phi}_{s, n}^\delta\right)\right|^2\right)\mathrm{d}s(x) \\
		&\le 3\left|\Gamma_m\right|\sum_{|n|\le N}\max_{x\in\Gamma_m}\sum_{\xi\in\{p, s\}}\!\!\left(\left|\kappa_{\xi, n}(r)\right|^2+\dfrac{n^2}{r^4}\left|\beta_{\xi, n}(r)\right|^2+\dfrac{n^2}{r^2}\left|\alpha_{\xi, n}(r)\right|^2\right)\left|\hat{\phi}_{\xi, n}-\hat{\phi}_{\xi, n}^\delta\right|^2 
	\end{align*}
		
		It follows from \eqref{eq:alpha_est} and \eqref{eq:kappa_est} that for $\xi\in\{p, s\},$
		\begin{align*}
			\left|\alpha_{\xi, n}(r)\right|^2&\le 2\left(k_\xi^2+\frac{n^2}{r^2}\right)\le 2\left(k_s^2+\frac{n^2}{r^2}\right),\\\left|\kappa_{\xi,n}(r)\right|^2&\le\frac{8n^4}{r^4}+\frac{2k_\xi^2}{r^2}\le\frac{8n^4}{r^4}+\frac{2k_s^2}{r^2},
		\end{align*}
		which further gives
		\begin{align}\nonumber
			&\quad \left\|\dfrac{\partial \bm{v}_N}{\partial\hat{x}}-\dfrac{\partial\bm{v}_N^\delta}{\partial\hat{x}}\right\|^2
			_{\left(L^2(\Gamma_m)\right)^2}\\\nonumber
		&\le 3|\Gamma_m|\left(\sum_{|n|\le 4 k_s\rho}+\sum_{4k_s\rho+1\le |n|\le N}\right) \\
		& \quad \cdot\max_{x\in\Gamma_m}\sum_{\xi\in\{p, s\}}\!\!\left(\left|\kappa_{\xi, n}(r)\right|^2+\dfrac{n^2}{r^4}\left|\beta_{\xi, n}(r)\right|^2+\dfrac{n^2}{r^2}\left|\alpha_{\xi, n}(r)\right|^2\right)\left|\hat{\phi}_{\xi, n}-\hat{\phi}_{\xi, n}^\delta\right|^2 \notag\\
		& \le 3|\Gamma_m|\left(\sum_{|n|\le 4 k_s\rho}+\sum_{4k_s\rho+1\le|n|\le N}\right)\sum_{\xi\in\{p, s\}}\left|\hat{\phi}_{\xi, n}-\hat{\phi}_{\xi, n}^\delta\right|^2 \notag\\
		&\quad\cdot\left\{\frac{8n^4}{(R+\gamma)^4}+\frac{2k_s^2}{(R+\gamma)^2}+\frac{n^2}{(R+\gamma)^4}+\frac{2n^2k_s^2}{(R+\gamma)^2}+\frac{2n^4}{(R+\gamma)^4}\right\} \notag \\
		&=:3|\Gamma_m|(I_3+I_4).\label{eq:dv_error}
	\end{align}
	where 
	\begin{align*}
		I_3 & =\sum_{|n|\le 4 k_s\rho}
		\left\{\frac{n^2(10n^2+1)}{(R+\gamma)^4}+\frac{2k_s^2(n^2+1)}{(R+\gamma)^2}\right\}
		\sum_{\xi\in\{p, s\}}\left|\hat{\phi}_{\xi, n}-\hat{\phi}_{\xi, n}^\delta\right|^2,\\
		I_4 & =\sum_{4k_s\rho+1\le|n|\le N}
		\left\{\frac{n^2(10n^2+1)}{(R+\gamma)^4}+\frac{2k_s^2(n^2+1)}{(R+\gamma)^2}\right\}\sum_{\xi\in\{p, s\}}\left|\hat{\phi}_{\xi, n}-\hat{\phi}_{\xi, n}^\delta\right|^2.
	\end{align*}
	
	Recalling \eqref{eq:phierr}--\eqref{eq:psierr}, we find that 
	\begin{align}\label{eq:I3}
		I_3\le(8k_s\rho+1)\left\{\frac{2560k_s^4\rho^4+16k_s^2\rho^2}{(R+\gamma)^4}+\frac{32k_s^4\rho^2+2k_s^2}{(R+\gamma)^2}\right\}\frac{M_2^2\delta^2\|\bm{v}\|^2_{(L^2(\Gamma_\rho))^2}}{\pi\rho}.
	\end{align}
	Meanwhile, noticing \eqref{eq:phierrorpsi} and the fact that $R+\gamma<\rho$, we derive that
	\begin{align}\nonumber
		I_4 & \le \sum_{4k_s\rho+1\le|n|\le N}\left\{\frac{(10n^2+1)n^2}{(R+\gamma)^4}+\frac{2k_s^2(n^2+1)}{(R+\gamma)^2}\right\}\frac{25\delta^2\ee^{2k_sR}\|\bmv\|^2_{(L^2(\Gamma_\rho))^2}}{\pi\rho^3}c_3n^2\tau_1^{2n} \\
		& \le\sum_{4k_s\rho+1\le|n|\le N}\left\{\frac{(10n^2+1)n^2}{(R+\gamma)^4}+\frac{2n^2(n^2+1)}{(R+\gamma)^4}\right\}\frac{25\delta^2\ee^{2k_sR}\|\bmv\|^2_{(L^2(\Gamma_\rho))^2}}{\pi\rho^3}c_3n^2\tau_1^{2n}\nonumber\\
		&=\sum_{4k_s\rho+1\le|n|\le N}\frac{12n^4+3n^2}{(R+\gamma)^4}\frac{25\delta^2\ee^{2k_sR}\|\bm{v}\|^2_{(L^2(\Gamma_\rho))^2}}{\pi\rho^3}c_3n^2\tau_1^{2n}\nonumber \\
		&\le\frac{375 c_3\ee^{2k_sR}\|\bm{v}\|^2_{(L^2(\Gamma_\rho))^2}}{\pi\rho^3(R+\gamma)^4}\sum_{4k_s\rho+1\le|n|\le N}n^6\tau_1^{2n}\delta^2\nonumber \\
		&\le\frac{750 c_3 \ee^{2k_sR}\|\bm{v}\|^2_{L^2(\Gamma_\rho)^2}}{\pi(R+\gamma)^4\rho^3}N^7\tau_1^{2N}\delta^2.\label{eq:I4}
	\end{align}
	
	Finally, collecting \eqref{eq:dv_error}, \eqref{eq:I3} and \eqref{eq:I4}, we obtain
	\begin{equation*}
		\left\|\dfrac{\partial \bm{v}_N}{\partial\hat{x}}-\dfrac{\partial\bm{v}_N^\delta}{\partial\hat{x}}\right\|^2
		_{\left(L^2(\Gamma_m)\right)^2}\le C_9\delta^2+C_{10}N^7\tau_1^{2N}\delta^2,
	\end{equation*}
	where
	\begin{align*}
		C_9 & =3(8k_s\rho+1)|\Gamma_m|\left\{\frac{2560k_s^4\rho^4+16k_s^2\rho^2}{(R+\gamma)^4}+\frac{32k_s^4\rho^2+2k_s^2}{(R+\gamma)^2}\right\}\frac{M_2^2\|\bmv\|^2_{(L^2(\Gamma_\rho))^2}}{\pi\rho},\\
		C_{10}&=\frac{750c_3|\Gamma_m|\ee^{2k_sR}\|\bm{v}\|^2_{L^2(\Gamma_\rho)^2}}{\pi(R+\gamma)^4\rho^3}.
	\end{align*}
\end{proof}

Before carrying on the convergence analysis, we shall take $N$ to be 
\begin{align}\label{eq:N}
	N=\left[\dfrac{\ln \delta^{-1}}{\ln\tau_2}\right],
\end{align}
where $[X]$ is the largest integer that is smaller than $X+1$ , then it holds that
\begin{align}\label{eq:err_rho_fixN}
	\left\|\bmv-\bmv^\delta_N\right\|^2_{\left(L^2(\Gamma_\rho)\right)^2}&\le C_1\delta^{2\left|\frac{\ln\tau_1}{\ln\tau_2}\right|}+C_2\tau_1^4\left|\frac{\ln\delta}{\ln\tau_2}\right|^4\delta^{2\left|\frac{\ln\tau_1}{\ln\tau_2}\right|}+C_3\delta^2+C_4\delta^2\frac{|\ln\delta|^4}{|\ln\tau_2|^4},\\\label{eq:error_v_fixN}
	\left\|\bmv-\bmv^\delta_N\right\|^2_{\left(L^2(\Gamma_m)\right)^2}&\le C_5\tau_2^4\dfrac{|\ln\delta|^2}{|\ln\tau_2|^2}\delta^2
	+C_6\delta^2+C_7\dfrac{|\ln\delta|^5}{|{\ln\tau_2}|^5}\delta^{2+2\left|\frac{\ln\tau_2}{\ln\tau_1}\right|},\\
	\label{eq:error_Deriv_fixN}\left\|\dfrac{\partial\bmv}{\partial\hat{x}}
	-\dfrac{\partial\bm{v}_N^\delta}{\partial\hat{x}}\right\|^2_{\left(L^2(\Gamma_m)\right)^2}&\le 
	C_8\delta^2\dfrac{|\ln\delta|^4}{|\ln\tau_2|^4}+
	C_9\delta^2+
	C_{10}\dfrac{|\ln\delta|^7}{|\ln\tau_2|^7}\delta^{2+2\left|\frac{\ln\tau_2}{\ln\tau_1}\right|}.
\end{align}

From \eqref{eq:err_rho_fixN}--\eqref{eq:error_Deriv_fixN}, we can easily see that
\begin{align}
	\left\|\bmv-\bmv^\delta_N\right\|_{\left(L^2(\Gamma_\rho)\right)^2} & \le\xi_1(\delta), \notag\\
	\label{eq:error_v_fixN1} \left\|\bmv-\bmv^\delta_N\right\|_{\left(L^2(\Gamma_m)\right)^2}&\le\xi_2(\delta),\\
	\left\|\dfrac{\partial\bmv}{\partial\hat{x}}
	-\dfrac{\partial\bm{v}_N^\delta}{\partial\hat{x}}\right\|_{\left(L^2(\Gamma_m)\right)^2}&\le \xi_3(\delta),\label{eq:error_Deriv_fixN1}
\end{align} 
with $\xi_\ell(\delta)\to 0,\,(\ell=1,2,3),$ as $\delta\to0.$

We first consider the Newton-type method where the noise-free data $\bm{v}$ is available to present the convergence results. Further, we shall establish convergence results of the method \eqref{eq:Newton}.\\

\noindent\textbf{Noise-free data}
As the first analysis, we assume that the given scattered field data $\bmv$ is noise-free. Let $p_m$ be the radial parameterization of the current approximation $\Gamma_m$ and $p$ be the exact boundary parameterization of the inverse problem. To consider the convergence, we assume that the solution $\bm{u}$ can be analytically extended to $\Gamma_m.$ Here, this analytical extension may be possibly through the interior of $D$. Within this extension, we can represent the exact scattered field $\bm{v}$ in the form of \eqref{eq:HD}, with $\phi_\xi,\,\xi\in\{p, s\}$ in \eqref{eq:HD} given in form of \eqref{eq:phi_psi_expansion} and the Fourier coefficients $\hat{\phi}_{\xi, n},\,\xi\in\{p, s\}$ being the solution to \eqref{eq:phipsiEq}.

Using the Newton-type method for the Dirichlet case, the step length $h$ at $m$-th step is determined by the linearized equation 
\[
\left.\Big(\bmu+\nabla\bmu\cdot h\Big)\right|_{\Gamma_m}=0.
\]
Within \Cref{assumption1} where $h(\hat{x})=h\hat{x},$ the linearized equation can be reduced to
\begin{align}\label{eq:linearized}
	\left.\left(\bmu+\dfrac{\partial\bmu}{\partial\hat{x}}h\right)\right|_{\Gamma_m}=0.
\end{align}
Under \Cref{assumption2}, the linearized equation \eqref{eq:linearized} can be solved pointwise by
\begin{align}\label{eq:h1}
	h=-\left.\left(\bmu\bigg/\dfrac{\partial\bmu}{\partial\hat{x}}\right)\right|_{\Gamma_m}.
\end{align}
The fact that $h$ is real-valued while the right-hand side may be complex-value motivates us to modify \eqref{eq:h1} by 
\begin{equation*}
	h=-\Re\left.\left(\bm{u}\bigg/\dfrac{\partial\bmu}{\partial\hat{x}}\right)\right|_{\Gamma_m}.
\end{equation*}
Meanwhile, the assumption \Cref{assumption2} must be replaced by
\begin{align*}
	\left|\Re\left(\left.\dfrac{\partial \bmu}{\partial \hat{x}}\right|_\Gamma\right)\right|>4\varepsilon.
\end{align*}
In addition, the update in \eqref{eq:Newton} is obtained by the following iterative scheme
\begin{align}\label{eq:iterscheme}
	p_{m+1}=p_m-\Re\left.\left(\bmu\bigg/\dfrac{\partial\bmu}{\partial\hat{x}}\right)\right|_{\Gamma_m},
\end{align}
where $\bmu$ is the exact solution of the direct scattering problem.

Combining the Dirichlet boundary condition in \eqref{eq:boundaryvalueproblem} and the Taylor's formula 
\begin{align*}
	\left|\bmu(p_m(\hat{x})\hat{x})-\bmu(p(\hat{x})\hat{x})-\dfrac{\partial\bmu(p_m(\hat{x})\hat{x}
		)}{\partial \hat{x}}(p_m(\hat{x})-p(\hat{x}))\right|=\mathcal{O}\big(|p_m(\hat{x})-p(\hat{x})|^2\big)
\end{align*}
gives
\begin{align*}
	\left\|\bmu|_{\Gamma_m}-\left.\dfrac{\partial\bmu}{\partial\hat{x}}\right
	|_{\Gamma_m}(p_m-p)\right\|_\infty=\mathcal{O}(\left\|p_m(\hat{x})-p(\hat{x})\right\|^2).
\end{align*}
Further, we obtain 
\begin{align*}
	p_{m+1}-p&=p_m-\left.\Re\left(\bmu\Bigg/\dfrac{\partial\bmu}{\partial\hat{x}}\right)\right|_{\Gamma_m}-p\\
	&=\Re\left[\left(\left.\left(p_m-p\right)\dfrac{\partial\bmu}{\partial\hat{x}}\right|_{\Gamma_m}
	-\bmu|_{\Gamma_m}\right)
	\Bigg/\left(\left.\dfrac{\partial\bmu}{\partial\hat{x}}\right|_{\Gamma_m}\right)\right],
\end{align*}
which implies that there exists some constant $C>0$ such that 
\[
\|p_{m+1}-p\|\le\dfrac{C}{4\varepsilon}\|p_m-p\|^2.
\]

\noindent\textbf{Noisy data}
In this part, we consider the Newton-type method by using the approximate scattered field $\bm{v}_N^\delta$.
In such a scenario, \eqref{eq:iterscheme} needs to be modified by
\begin{align}\label{eq:iterscheme_delta}
	p_{m+1}^\delta=p_m^\delta-\Re\left.\left(\bm{u}_N^\delta\bigg/\dfrac{\partial\bmu_N^\delta}{\partial\hat{x}}\right)\right|_{\Gamma_m^\delta},
\end{align}
where $\bmu^\delta_N=\bmu^i+\bmv^\delta_N,$ and we denote $p_m^\delta$ and $\Gamma_m^\delta$ to indicate the dependence on the noise.

We are now in a position to head for the convergence results.
Noticing  \eqref{eq:lowboundepsilon}, \eqref{eq:error_v_fixN},  and \eqref{eq:error_Deriv_fixN},  we derive that there exist a sufficiently small closed neighborhood $U$ of $\Gamma$ and a sufficiently small $\delta$ such that 
\begin{align}\label{eq:derivBound}
	\left|\left.\dfrac{\partial\bm{u}_N^\delta}{\partial\hat{x}}\right|_{\Gamma}\right|\ge 2\varepsilon,\quad \Gamma\in U.
\end{align}

To establish the convergence of the method, a stopping rule is supposed to be imposed to the Newton-type iterative scheme:

	Given the measured noisy data $\bmv^\delta$ by \eqref{eq:noisyData}, we stop the iterative process \eqref{eq:Newton} if two successive approximations satisfy
	\begin{equation}\label{eq:stop}
		\|p_{m+1}^\delta-p_m^\delta\|_{L^2(\mathbb{S})}\le C_1(\delta),
	\end{equation}
	where $C_1=\dfrac{4\xi_2(\delta)}{\varepsilon- 2\xi_3(\delta)}.$

\begin{remark}
	We would like to point out that, $C_1(\delta)\to0$ as $\delta\to0,$ which means that the stopping rule gets more strict as $\delta$ decreases.
\end{remark}

For convergence analysis, we define $p_\delta:=p_{m+1}^\delta$ when the stopping criterion is fulfilled.
\begin{theorem}[Convergence Result]
	Let $\Gamma$ be analytic and assume that \eqref{eq:derivBound} holds. Then the iterative scheme \eqref{eq:Newton} is locally convergent in the sense that
	\[
	\|p_\delta-p\|\to0,\quad \text{as }\delta\to0.
	\]
\end{theorem}

\begin{proof}
	We start this proof by noticing that 
	\begin{align*}
		&\quad p_{m+1}^\delta-p \\
		& =p_m^\delta-\Re\left(\bm{u}_N^\delta\left.\Bigg/\dfrac{\partial\bmu_N^\delta}
		{\partial\hat{x}}\right)\right|_{\Gamma_m^\delta}-p \\
		& =\Re\left[\left(-\bm{u}_N^\delta|_{\Gamma_m^\delta}+(p_m^\delta-p)\left.\dfrac{\partial\bmu_N^\delta}
		{\partial\hat{x}}\right|_{\Gamma_m^\delta}\right)\Big/\left(\left.\dfrac{\partial\bm{u}_N^\delta}
		{\partial\hat{x}}\right|_{\Gamma_m^\delta}\right)\right] \\
		& = \Re\left\{\left[(p_m^\delta-p)\left.\dfrac{\partial\bmu^\delta}{\partial\hat{x}}\right|_{\Gamma_m^\delta}
		\left.-\bmu_N^\delta\right|_{\Gamma_m^\delta}\right.\right.\left.\left.+(p_m^\delta-p)\left.\left(
		\dfrac{\partial\bm{u}_N^\delta}{\partial\hat{x}}-\dfrac{\partial\bmu^\delta}{\partial\hat{x}}\right)\right|_{\Gamma_m^\delta
		}\right]\Big/\left(\left.\dfrac{\partial\bm{u}_N^\delta}{\partial\hat{x}}\right|_{\Gamma_m^\delta}\right)\right\}.
	\end{align*}
	Here, for $\Gamma_m^\delta\in U,$ it holds that
	\begin{align*}
		\left|\left.(p_m^\delta-p)\dfrac{\partial \bmu^\delta}{\partial\hat{x}}\right.-\left.\bmu_N^\delta\right.\right|&\le
		\left|(p_m^\delta-p)\dfrac{\partial \bmu^\delta}{\partial\hat{x}}-\left.\bmu^\delta\right.\right|+\left|\left.\bmu^\delta\right.\left.-\bm{u}_N^\delta\right.\right|\\
		&\le\left|(p_m^\delta-p)\dfrac{\partial \bmu^\delta}{\partial\hat{x}}-\bmu^\delta\right|+\left|\bmu-\bmu^\delta\right|+\left|\bmu-\bmu^\delta_N\right|,
	\end{align*}
	which further gives
	\begin{equation*}
		\left\|\left.(p_m^\delta-p)\dfrac{\partial \bmu^\delta}{\partial\hat{x}}\right.-\left.\bm{u}_N^\delta\right.\right\|
	 \le c\|p_m^\delta-p\|^2+\delta\|u\|+\xi_2(\delta)\le c\|p_m^\delta-p\|^2+\xi_4(\delta),
	\end{equation*}
	where $c>0$ is a constant depending on $\bmu^\delta,$ $\xi_4(\delta)=\delta\|u\|+\xi_2(\delta)$ and $\xi_4(\delta)\to0$ as $\delta\to0.$
	Combining the above fact with \eqref{eq:error_v_fixN1}, \eqref{eq:error_Deriv_fixN1}, \eqref{eq:derivBound} and Taylor's formula gives the following result
	\begin{align}\label{eq:En}
		\|p_{m+1}^\delta-p\|\le\dfrac{1}{\varepsilon}\left(c\|p_m^\delta-p\|^2+\xi_3(\delta)\|p_m^\delta-p\|+\xi_4(\delta)\right).
	\end{align}
	
	In fact, \eqref{eq:En} establishes the error relation between the $(m+1)$-th iteration and the $m$-th iteration. To ensure the convergence, it suffices to impose a condition such that the following inequality holds:
	\begin{align}\label{eq:inequality}
		\dfrac{1}{\varepsilon}\left(c\|p_m^\delta-p\|^2+\xi_3(\delta)\|p_m^\delta-p\|+\xi_4(\delta
		)\right)\le\frac{1}{2}\|p_{m}^\delta-p\|.
	\end{align}
	From the solution formula for quadratic equations, we can see that to ensure \eqref{eq:inequality}, $\|p_m^\delta-p\|$ is supposed to satisfy:
	\begin{align*}
		\frac{\beta-\sqrt{\beta^2-4c\xi_4(\delta)}}{2c}<\|p_m^\delta-p\|<\frac{\beta+\sqrt{\beta^2-4c\xi_4(\delta)}}{2c},
	\end{align*}
	where $\beta:=\varepsilon/2-\xi_3(\delta)$. From the Talyor's expansion of the square root function around $\beta,$ we know that
	for a sufficiently small fixed $\delta,$ if
	\begin{align}\label{eq:require}
		C_1(\delta)=\dfrac{4\xi_4(\delta)}{\varepsilon- 2\xi_3(\delta)}<\|p_m^\delta-p\|\le\dfrac{\varepsilon- 2\xi_3(\delta)}{2C},
	\end{align}
	it holds that 
	\begin{align}\label{eq:2}
		\|p_{m+1}^\delta-p\|\le\dfrac{1}{\varepsilon}\left(c\|p_m^\delta-p\|^2+\xi_3(\delta)\|p_m^\delta-p\|+\xi_3(\delta
		)\right)\le\dfrac{\|p_m-p\|}{2}.
	\end{align}
	In other words, for a sufficiently small $\delta,$ once $p_m^\delta$ satisfies \eqref{eq:require}, then \eqref{eq:2} holds. In addition, under \eqref{eq:require}, it holds that
	\begin{equation*}
		\|p_{m+1}^\delta-p_m^\delta\|=\|p_{m+1}^\delta-p-(p_m^\delta-p)\|\ge\|p_{m+1}^\delta-p\|-\|p_m^\delta-p\|\ge\dfrac{1}{2}\|p_m^\delta-p\|.
	\end{equation*}
	Further, once $p_m^\delta$ satisfies the stopping rule \eqref{eq:stop}, then 
	\begin{align*}
		\|p_{m+1}^\delta-p\|\le\dfrac{1}{2}\|p_m^\delta-p\|\le\|p_{m+1}^\delta-p_m^\delta\|\le C_1(\delta).
	\end{align*}
	By definition of $p_\delta,$ we have that
	\begin{align*}
		\|p_\delta-p\|\le C_1(\delta)\to0
	\end{align*}
	as $\delta\to0,$ which completes the proof.
\end{proof}
	
	
\section{Numerical examples}\label{sec:example}

In this section, we shall provide several numerical examples to verify the performance of the Newton-type method proposed in the former sections. Let $\gamma_m,$ parameterized by $p_m,$ be the current approximation to the exact boundary curve $\Gamma,$ that is, 
\begin{equation*}
	p_m(t)=\{r_m(t)(\cos t,\sin t):t\in[0,2\pi]\},
\end{equation*}
where $r_m:\mathbb{S}\to\mathbb{R}_+$ is the radial function. Here, we choose the radial functions $r_m$ to be the linear combinations of trigonometric polynomials of degree less than or equal to $N_p,$ that is,
\begin{equation}\label{eq:radial}
	r_m(t)=a_0+\sum_{j=1}^{N_p}(a_j^m\cos j t+b_j^m\sin j t),
\end{equation}
with $a_j^m, b_j^m\in\mathbb{R},\,j=1,\cdots, N_p, m=1,2,\cdots.$

Given the current approximation $r_m,\,m=1,2,\cdots,$ of the boundary $\partial D,$ the  approximation $r_{m+1}$ at the $(m+1)$-th step is updated through 
\[
r_{m+1}=r_m+\Delta r_m,
\]
where $\Delta r_m$ is the update step length in the form of \eqref{eq:radial} determined by \eqref{eq:iterscheme_delta} with  Fourier coefficients $\Delta c_m=(\Delta a_0^m, \Delta a_1^m,\cdots,\Delta a_{N_p}^m,\Delta b_1^m,\cdots,\Delta b_{N_p}^m),j=1,\cdots, N_p,\,m=1,2,\cdots$. In practical computation, once the initial guess is given, to solve \eqref{eq:iterscheme_delta} is equivalent to seeking for $\Delta c_m$ that is determined by the first line in \eqref{eq:Newton}. Explicitly, the Fourier coefficients $\Delta c_m$ are supposed to satisfy
\begin{align*}
	\bm{u}_N^\delta(Bc_m\hat{x})+(\nabla^\top\bm{u}_N^\delta(Bc_m\hat{x})\hat{x})B\Delta c_m=0,
\end{align*}
where $\hat{x}(t)=(\cos t,\sin t)^\top,$ $B(t)=(1,\cos t,\cdots, \cos N_r t,\sin t,\cdots,\sin N_r t)$, and then the update $c_{m+1}=c_m+\Delta c_m$ is obtained.

To terminate the iteration, we choose some error tolerance $\epsilon>0$ in advance and compute the relative error 
\[
e_M=\dfrac{\|\Delta r_m\|}{\|r_m\|}.
\]
Once $e_M<\epsilon,$ we stop the iteration and the current approximation $r_m$ is served as the final reconstruction of $\partial D.$

For the numerical examples in this section, the synthetic scattered fields are generated by solving the forward elastic scattering problem \eqref{eq:boundaryvalueproblem} via boundary integral equation method \cite{Dong1}. The synthetic scattered field data are numerically generated with a numerical quadrature rule with 64 points equidistantly distributed on $[0,2\pi].$ The $128$ receivers $x_i\in\Gamma_\rho$ are chosen to be $x_i=3(\cos\theta_i,\sin\theta_i)$ with $\theta_i\in[0,2\pi]$ being the observation angle. The 20 point sources are equally distributed on the measurement circle, and the sources are marked by red small points in all the figures.

To test the stability of the proposed method, we add random noise to the measured data. The noisy scattered field data are given according to the following formula:
\begin{align*}
	\bm{v}^\epsilon=\bm{v}+\delta r_1|\bm{v}|\mathrm{e}^{\mathrm{i}\pi r_2},
\end{align*}
where $r_1,\,r_2$ are two uniformly distributed random numbers ranging from $-1$ to $1,$ and $\delta>0$ is the noise level. In all the numerical examples, $\delta$ is chosen to be $5\%.$ To stop the Newton method, $\epsilon$ is selected to be $10^{-4}$. The radial functions $r_m$ are given in the form of \eqref{eq:radial} with $N_p$ chosen to be $8$ without loss of generality.

In the following examples, we take the Lam\'{e} constants to be $\lambda=\mu=1,$ thus the wavenumbers are $k_p=\omega/\sqrt{3}$ and $k_s=\omega.$ The polarization is taken to be ${\bm p}=(\frac{\sqrt{2}}{2},\frac{\sqrt{2}}{2}).$ According to \eqref{eq:N}, we used the following rule for choosing $N:$
\begin{align*}
	N:=2\left[|\ln\delta|\right]+1.
\end{align*}

In the following figures exhibiting the reconstructions, the exact boundary curves are displayed by the red solid lines, the reconstructions are indicated by the black dash-dotted lines $\cdot-,$ the green dashed lines mark the initial guess and the receivers are located on the blue dashed lines. To visualize the aperture of measurement, we additionally plot the auxiliary curve by black dashed lines to show the data aperture.

\begin{example}
	In the first example, we consider the reconstruction of the kite-shaped obstacle, whose boundary is parameterized through
	$$
	x(t)=(\cos t+0.65\cos2t-0.65,1.5\sin t),\quad 0\le t\le 2\pi.
	$$
	
	In \Cref{fig:kite}, we exhibit the reconstruction of the obstacle subject to $\omega=5$  under different observation apertures. We display the geometry setup in \Cref{fig:kite}(a). In \Cref{fig:kite}(b), we show the reconstruction by using the full aperture data. We can easily see from \Cref{fig:kite}(b) that the kite-shaped obstacle is well-reconstructed when the observation points are located on the full aperture. Next, we consider the limited-aperture case in \Cref{fig:kite}(c)--\Cref{fig:kite}(i) by taking different observation angles or placing the observation points in different positions. It can be observed that when the observation angles cover the concave part of the kite-shaped obstacle, the obstacle is well-reconstructed and the reconstruction is acceptable even though the observation points distribute only on $1/4$-aperture angle! For the limited observation where the concave part of the obstacle is partially located in the `shadow region', the method fails to reconstruct the concave part which is illuminated inadequately. Nevertheless, the concave part located in the illuminated domain as well as the convex part can be well-reconstructed.
\begin{figure}
	\centering
	\subfigure[]{\includegraphics[width=0.32\textwidth]{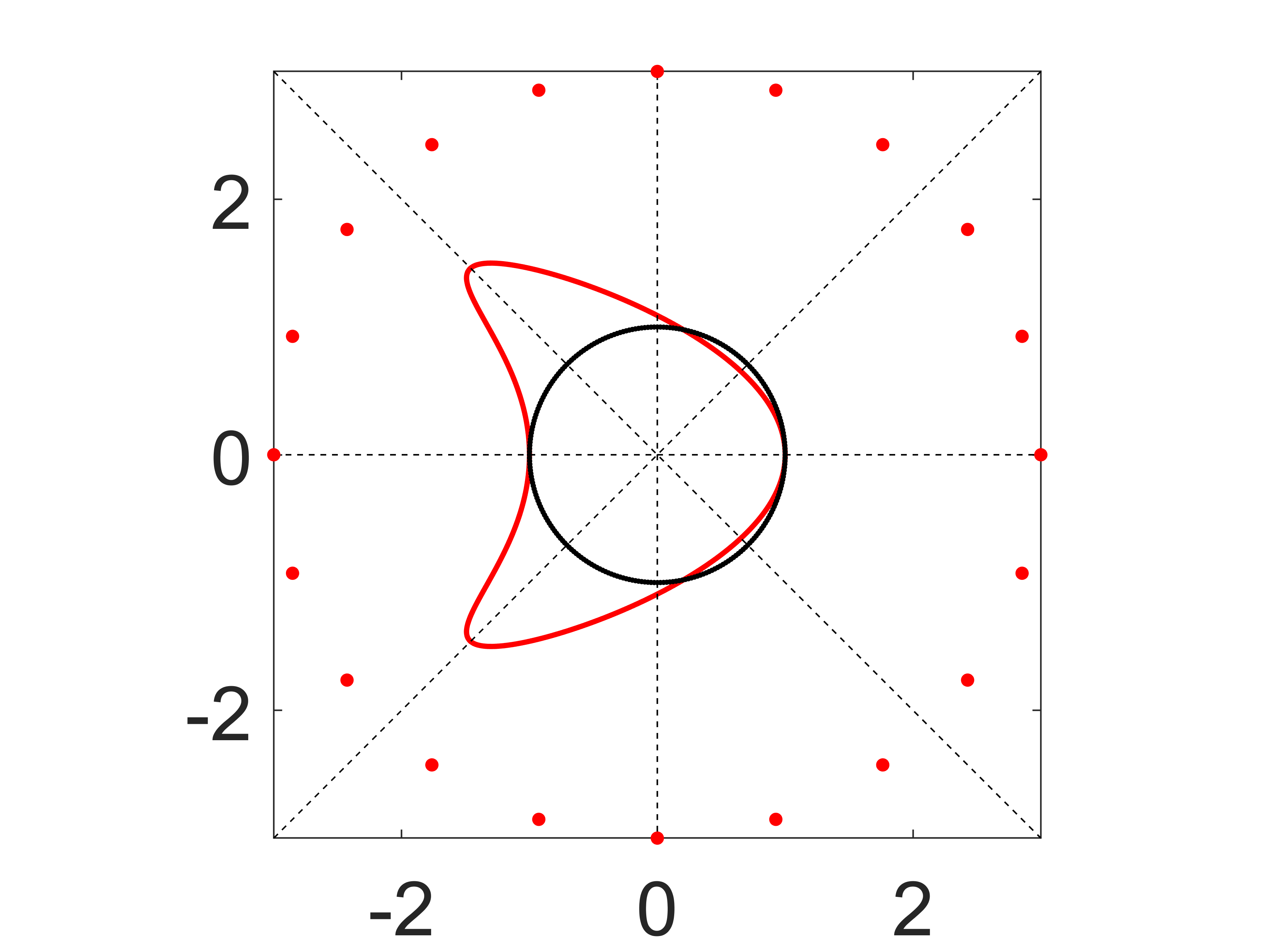}}
	\subfigure[]{\includegraphics[width=0.32\textwidth]{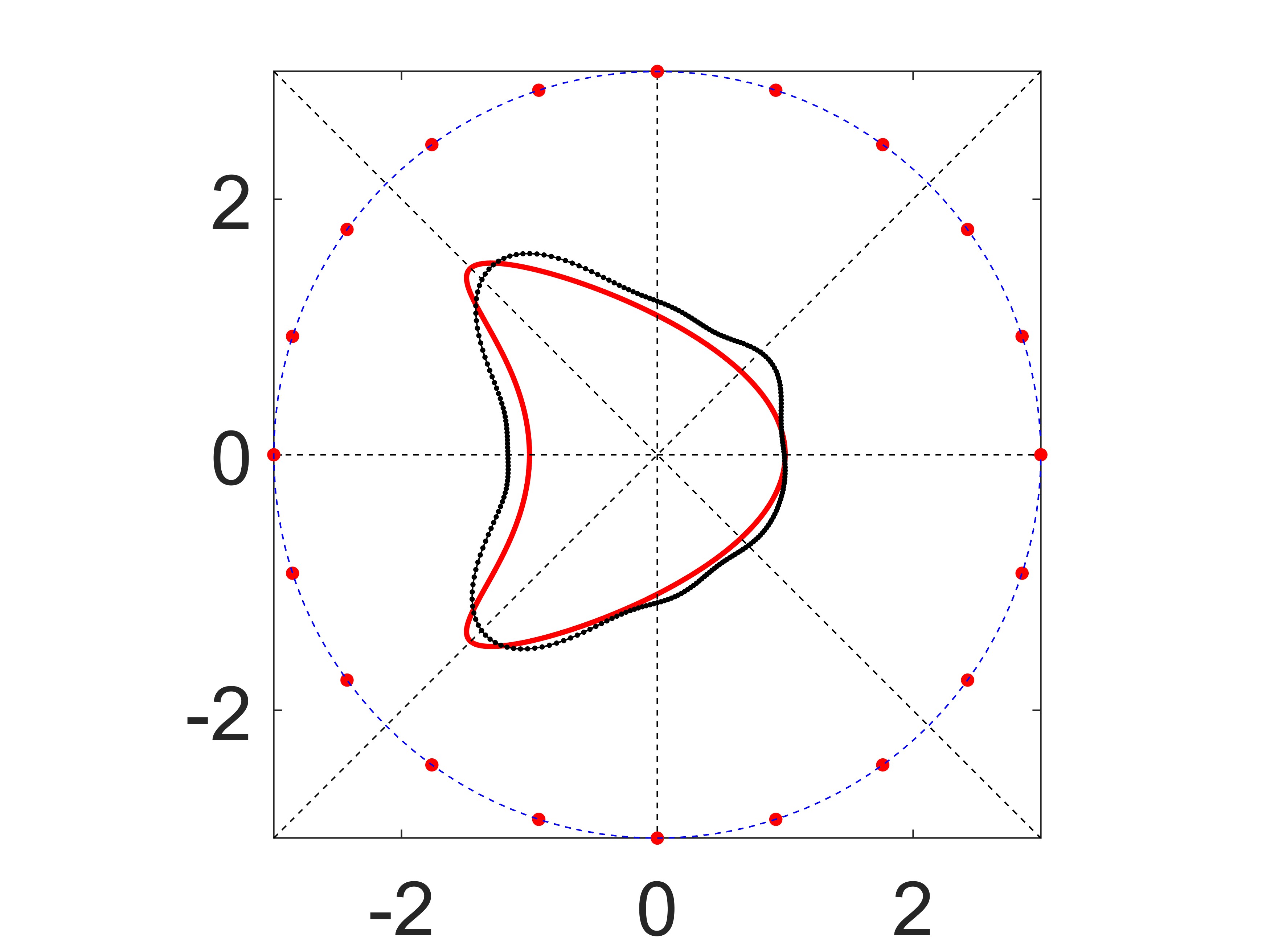}}
	\subfigure[]{\includegraphics[width=0.32\textwidth]{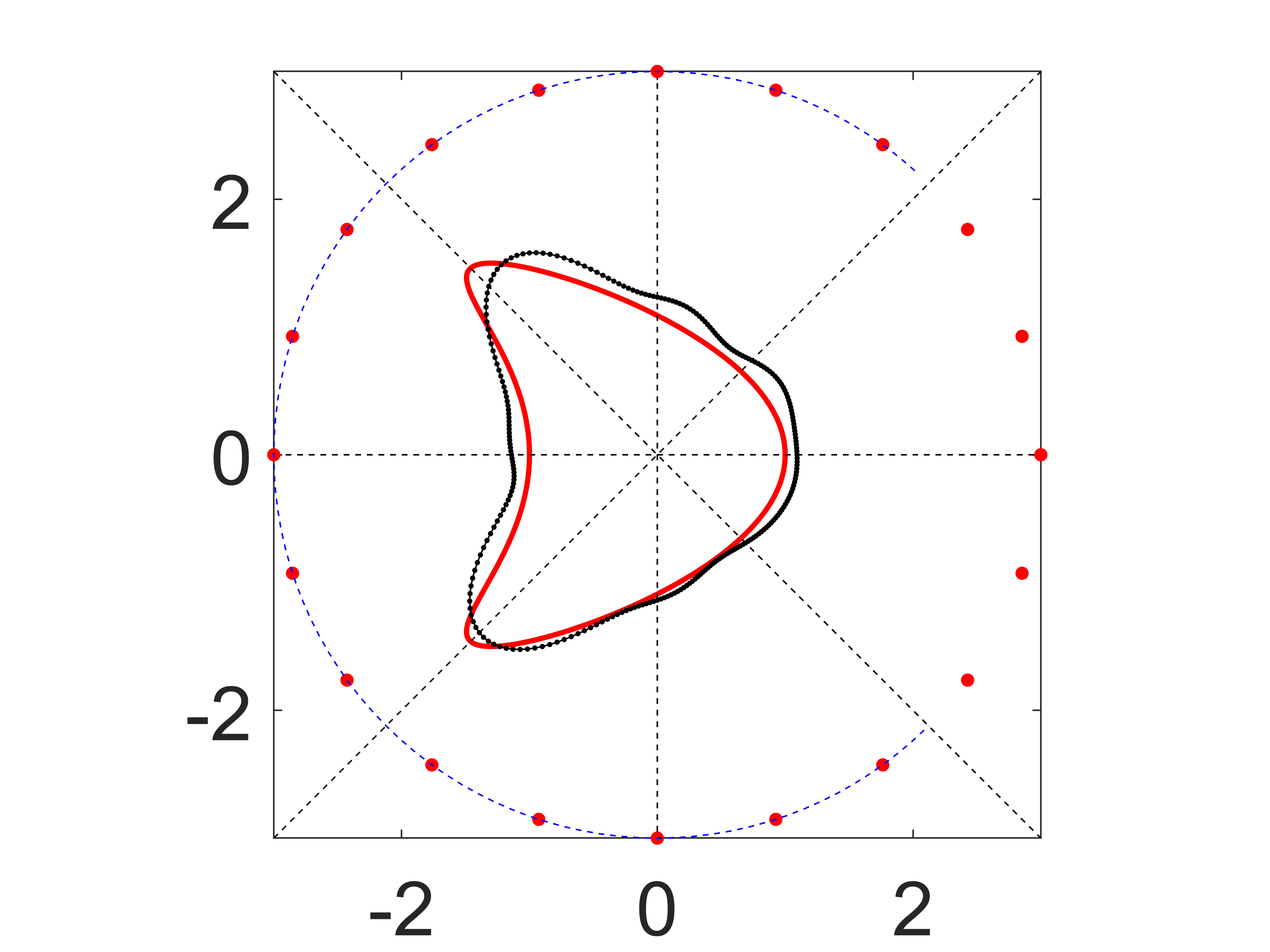}}
	\subfigure[]{\includegraphics[width=0.32\textwidth]{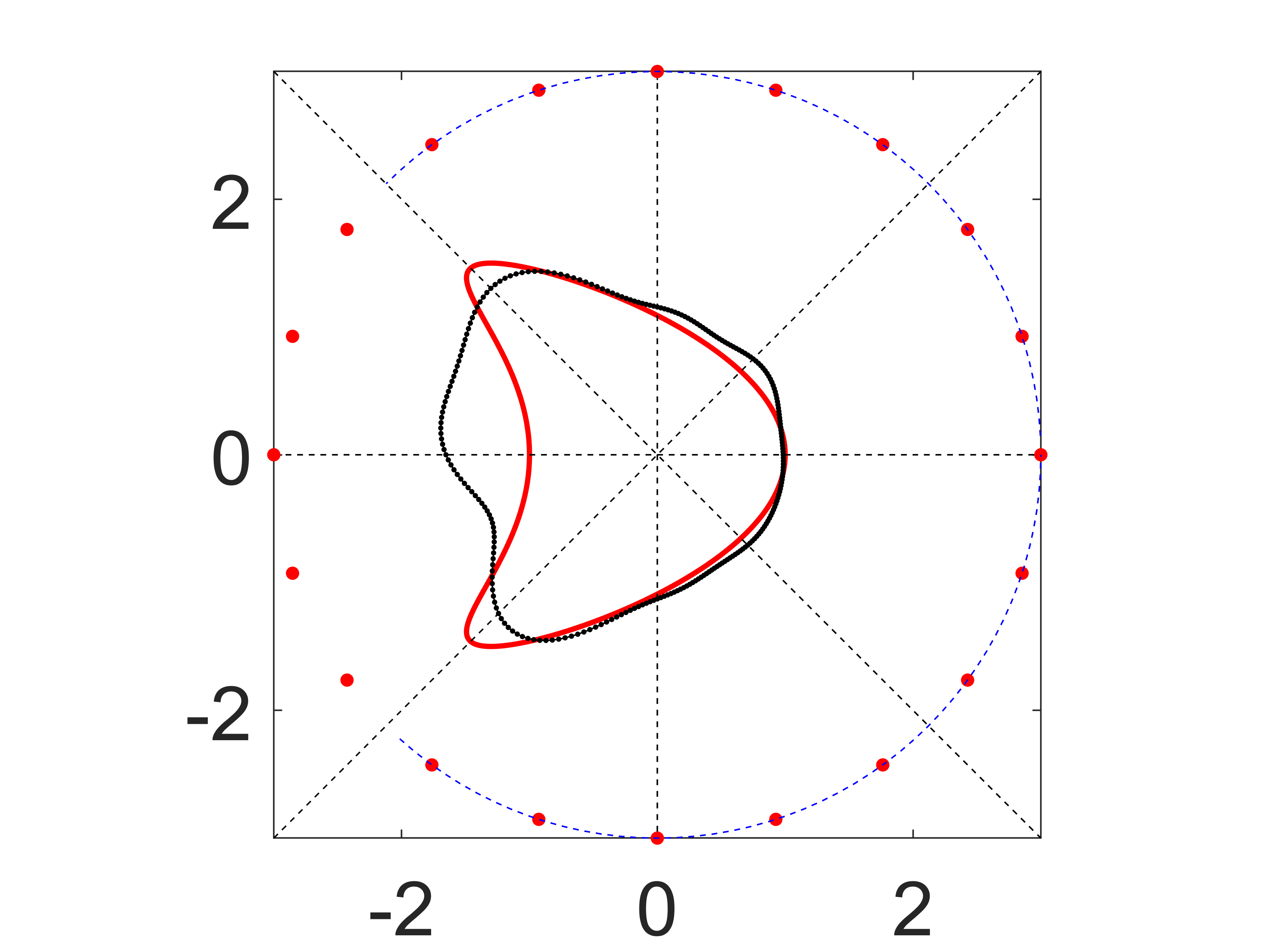}}
	\subfigure[]{\includegraphics[width=0.32\textwidth]{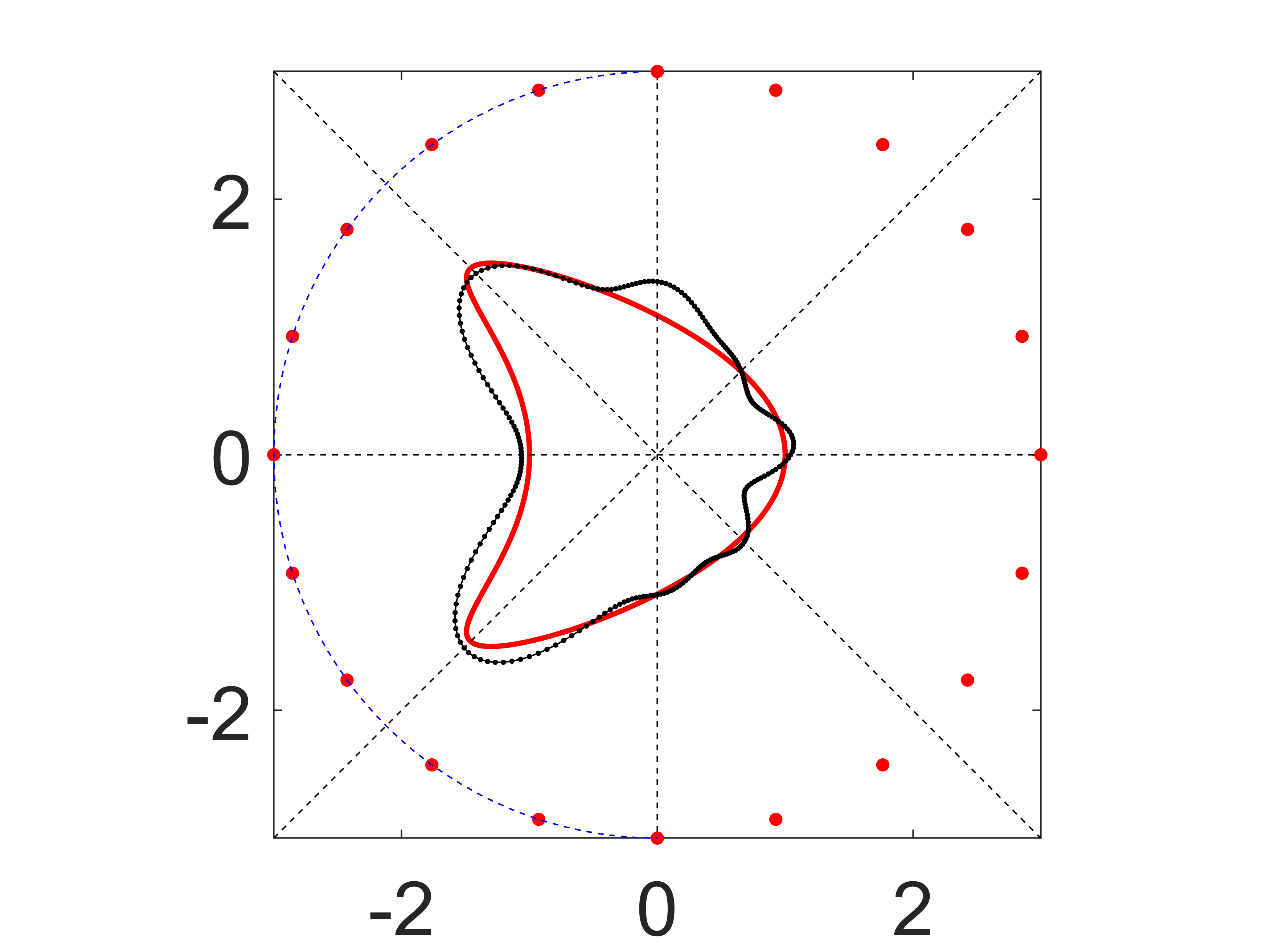}}
	\subfigure[]{\includegraphics[width=0.32\textwidth]{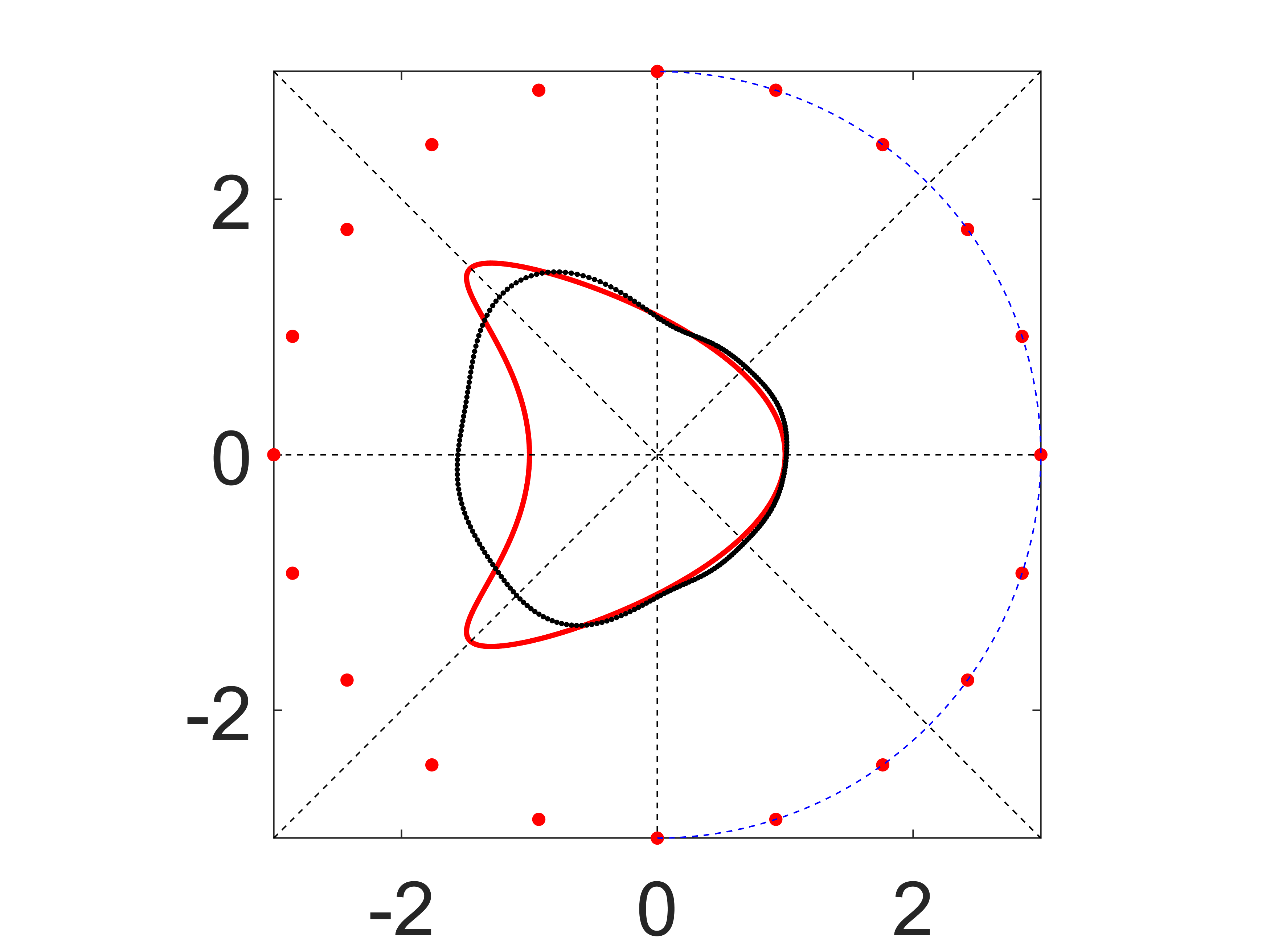}}
	\subfigure[]{\includegraphics[width=0.32\textwidth]{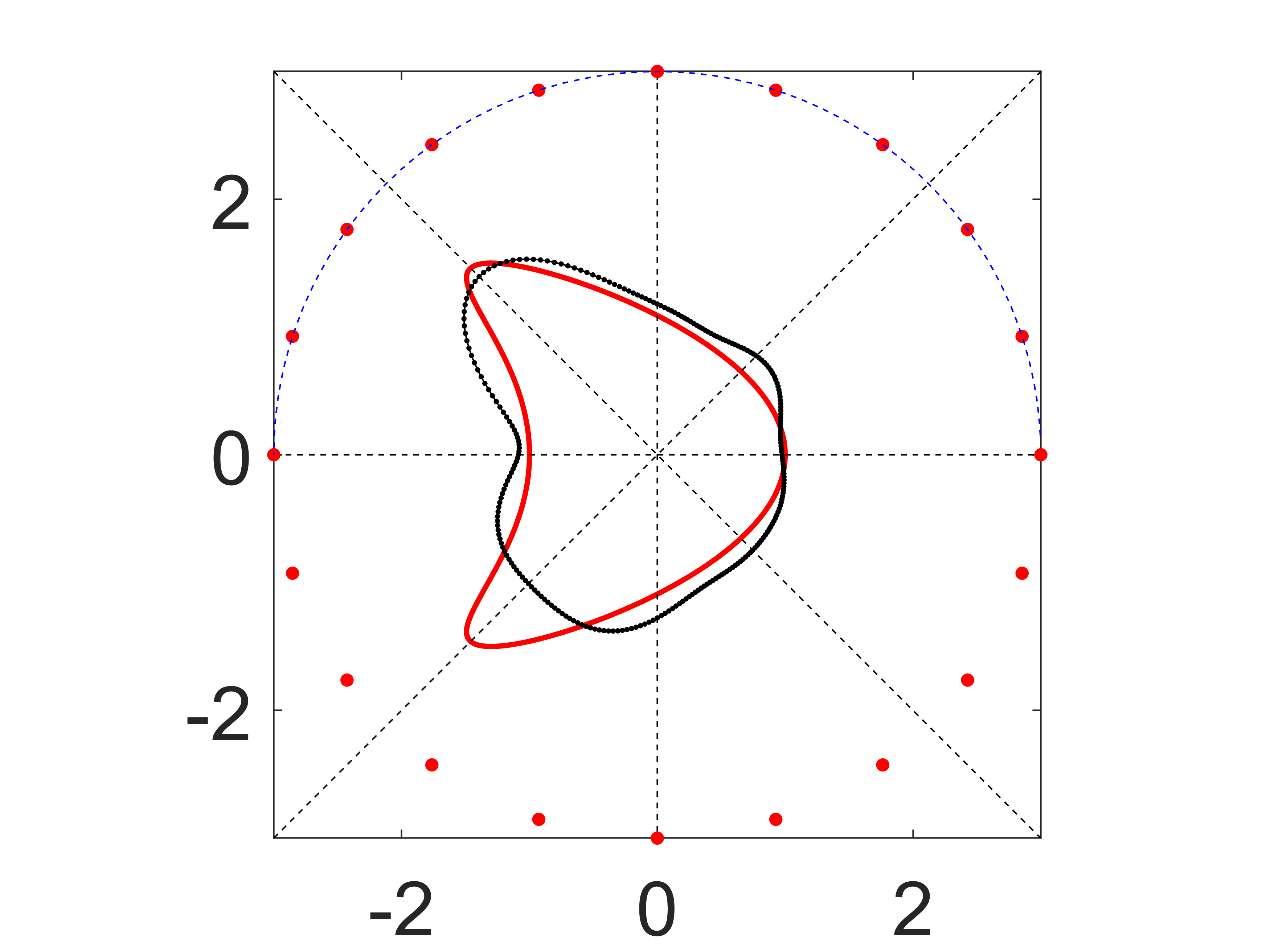}}
	\subfigure[]{\includegraphics[width=0.32\textwidth]{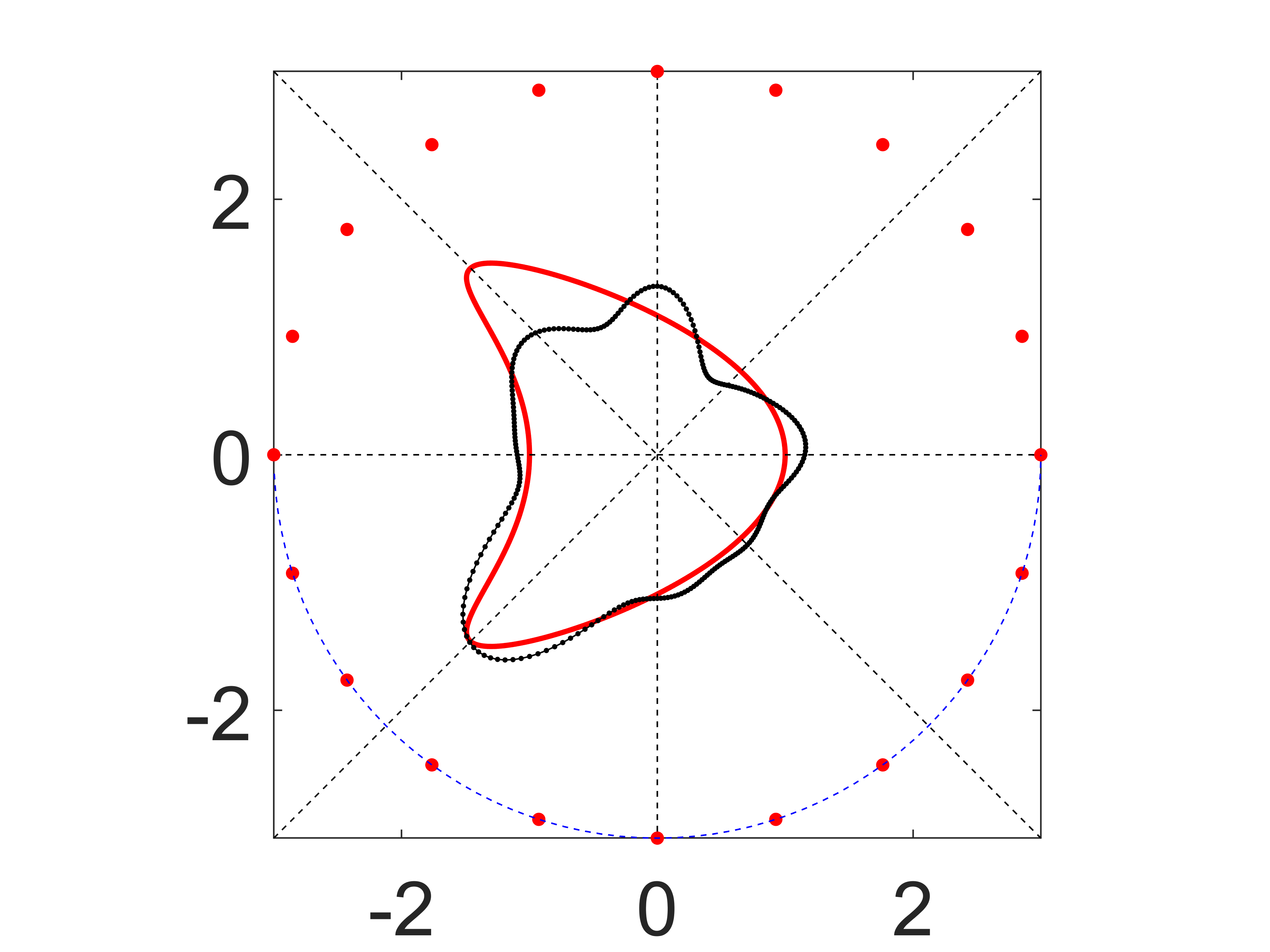}}
	\subfigure[]{\includegraphics[width=0.32\textwidth]{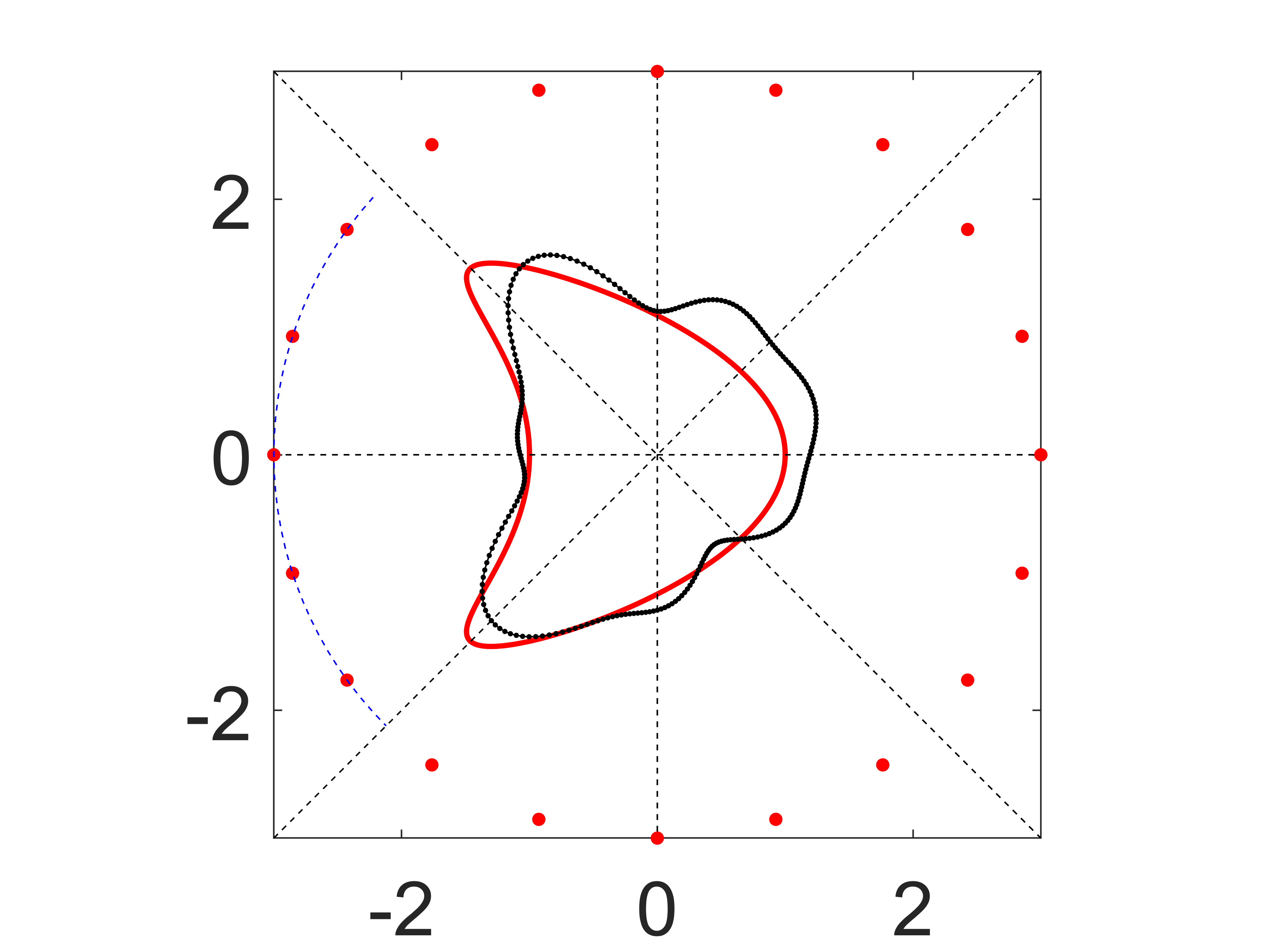}}
	\caption{Reconstruction of the kite-shaped obstacle subject to different observation apertures. (a) exact obstacle and initial guess; (b) $\theta_i\in\left[0,2\pi\right);$ (c) $\theta_i\in\left[{\pi}/{4},{7\pi}/{4}\right);$ (d) $\theta_i\in\left[0,{3\pi}/{4}\right)\cup\left[{5\pi}/{4},2\pi\right);$ (e)  $\theta_i\in\left[{\pi}/{2},{3\pi}/{2}\right);$ (f) $\theta_i\in\left[0,{\pi}/{2}\right)\cup\left[{3\pi}/{2},2\pi\right);$ (g)  $\theta_i\in\left[0,\pi\right);$ (h) $\theta_i\in\left[\pi,2\pi\right);$ (i) $\theta_i\in\left[{3\pi}/{4},{5\pi}/{4}\right).$}\label{fig:kite}
\end{figure}
\end{example}
\begin{example}
	In this example, we consider the reconstruction of the starfish-shaped obstacle whose boundary can be parameterized by 
	$$
	x(t)=(1.5\cos t+0.15\cos 4t+0.15\cos6t,1.5\sin t-0.15\sin4t+0.15\sin6t).
	$$
	
	\Cref{fig:starfish} shows the reconstruction of the starfish under different observation apertures with $\omega=5.$ In \Cref{fig:starfish}(a), we display the geometry and the initial guess. We can easily see that the obstacle can be well-reconstructed under full-aperture observation. When only limited aperture data is available, those parts under the illuminated domain can be recovered satisfactorily. 
	
Further, we change the initial guess and reconsider the reconstruction of the starfish in \Cref{fig:starfish2}. By choosing different initial guesses, we find that our method can always achieve a sastisfactory reconstruction. We would like to point out that, as analyzed in \Cref{sec:convergence}, the Newton method has only the local convergence, so we may not expect a further global convergence. Nevertheless, it is not difficult to select a proper initial guess. As can be seen in \Cref{fig:starfish2}, though with different initial guesses, the obstacle can all be well-reconstructed.

\begin{figure}
	\centering  
	\subfigure[]{\includegraphics[width=0.32\textwidth]{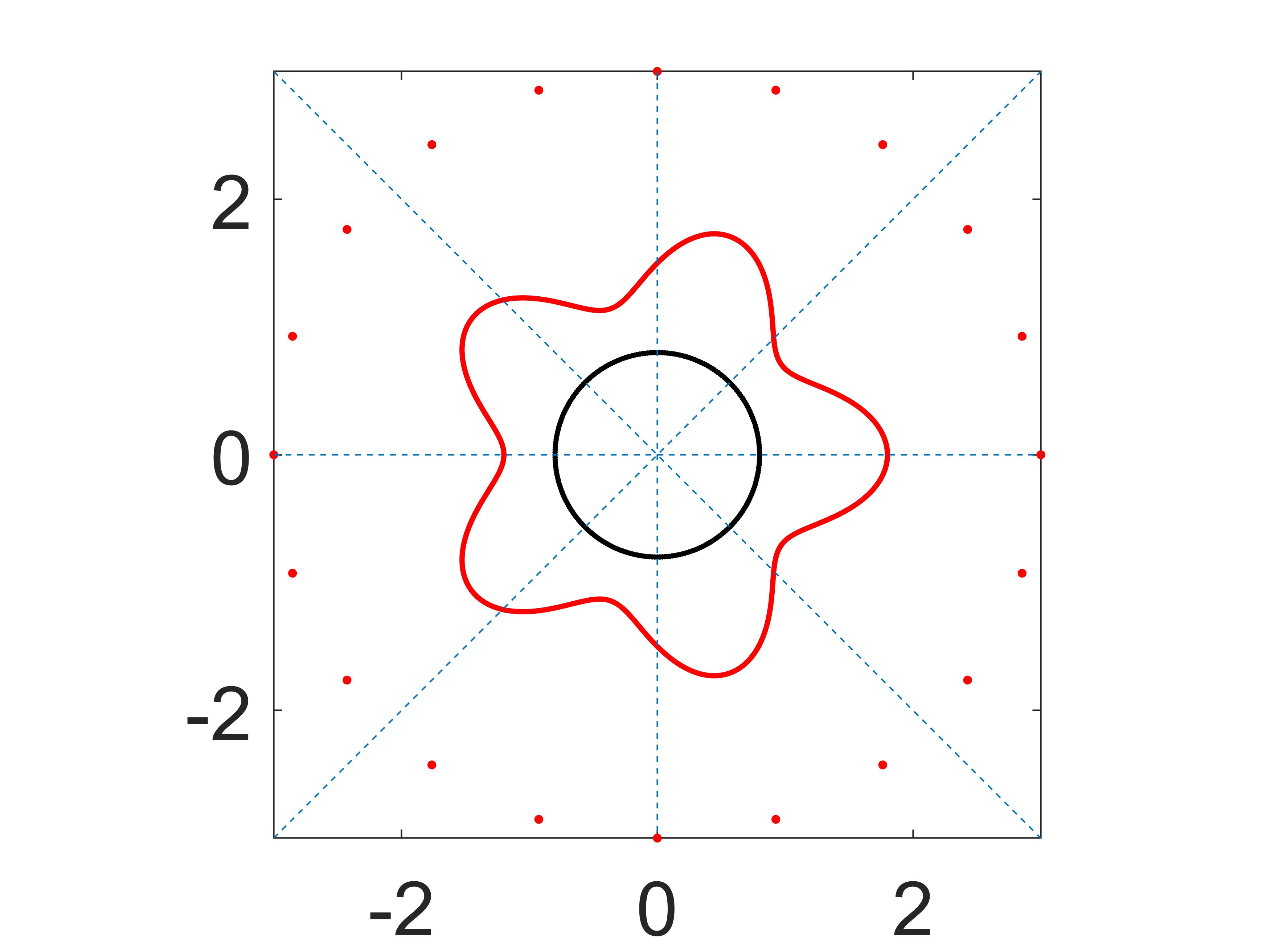}}
	\subfigure[]{\includegraphics[width=0.32\textwidth]{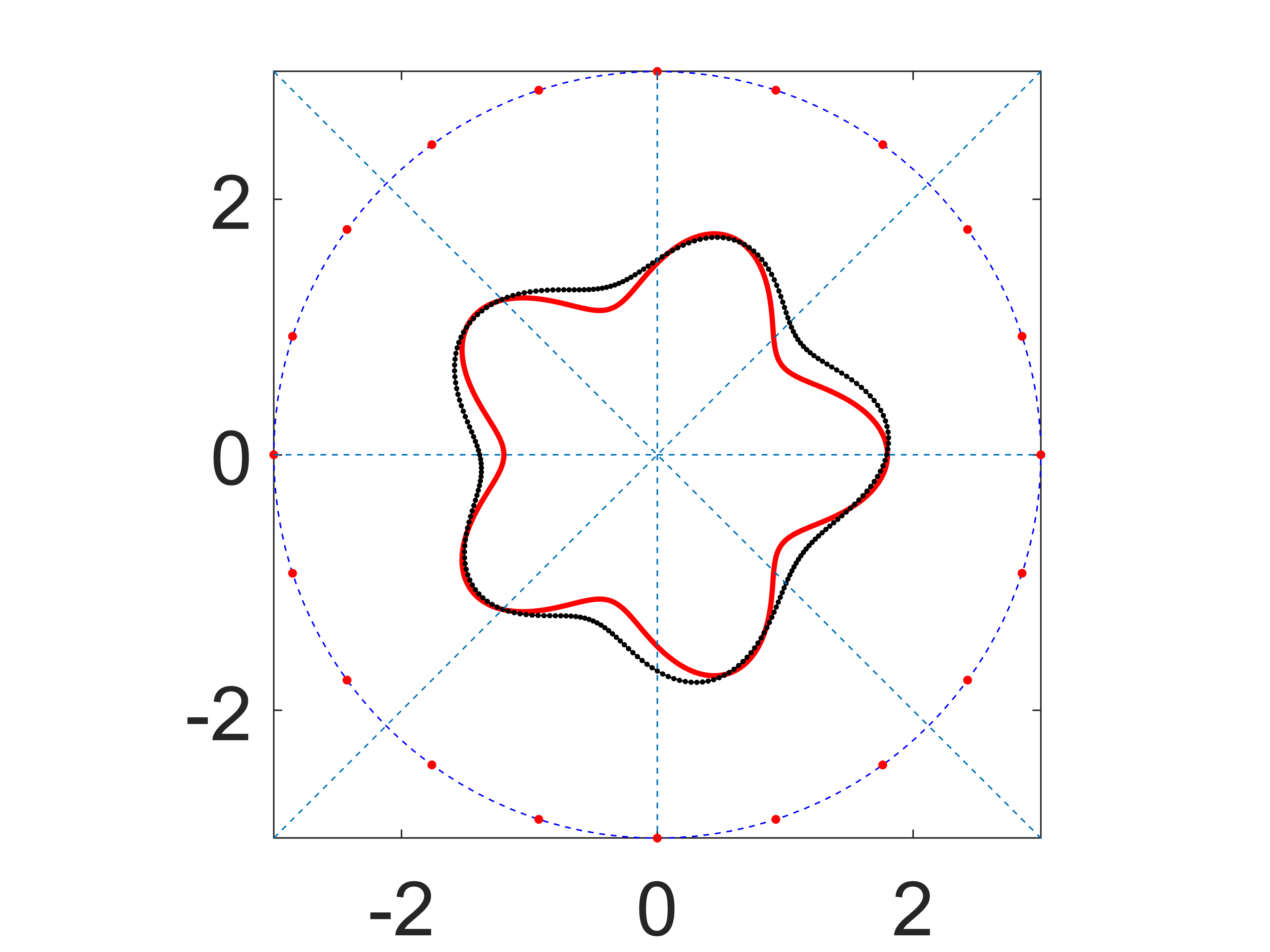}}
	\subfigure[]{\includegraphics[width=0.32\textwidth]{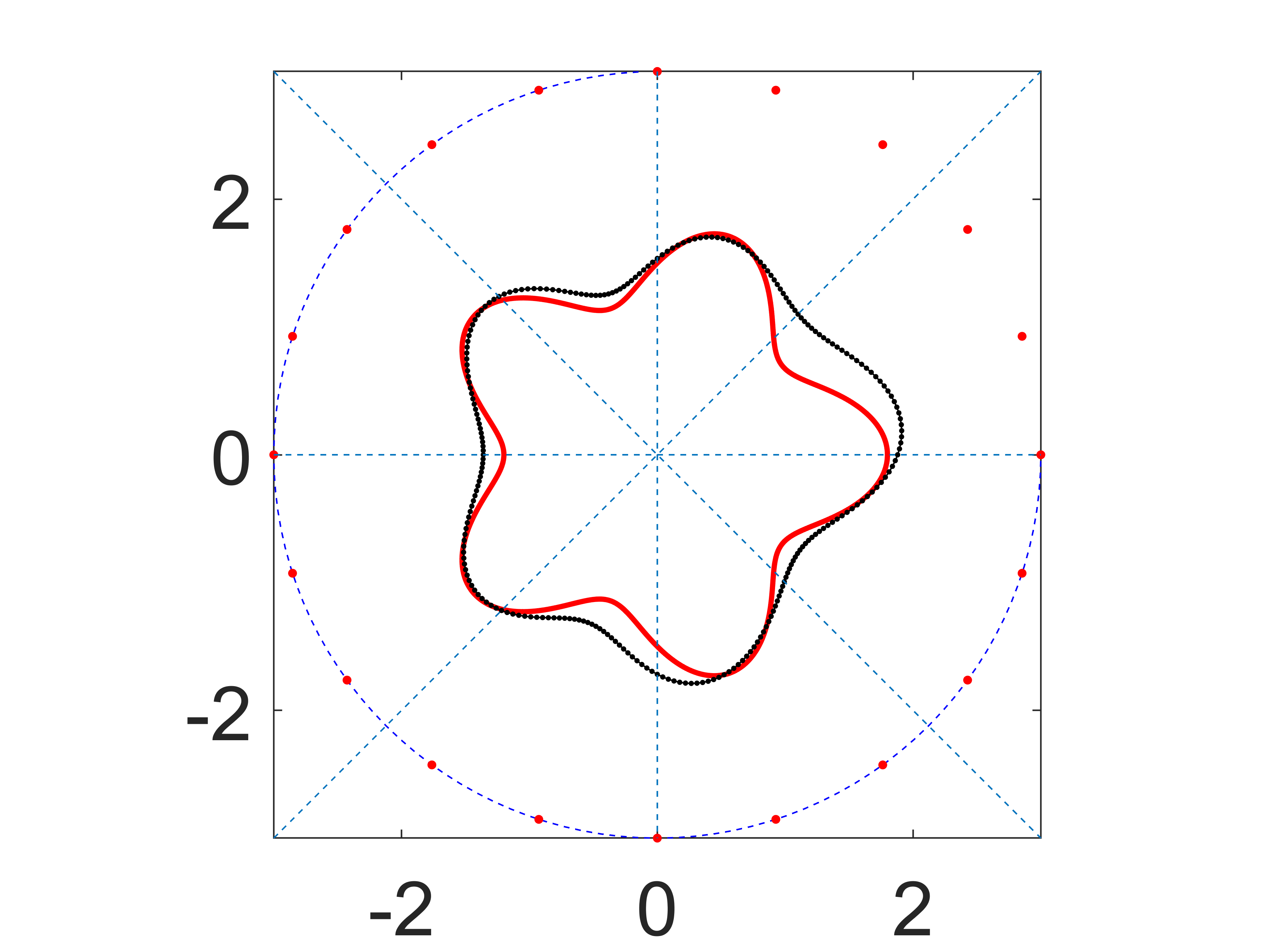}}
	\subfigure[]{\includegraphics[width=0.32\textwidth]{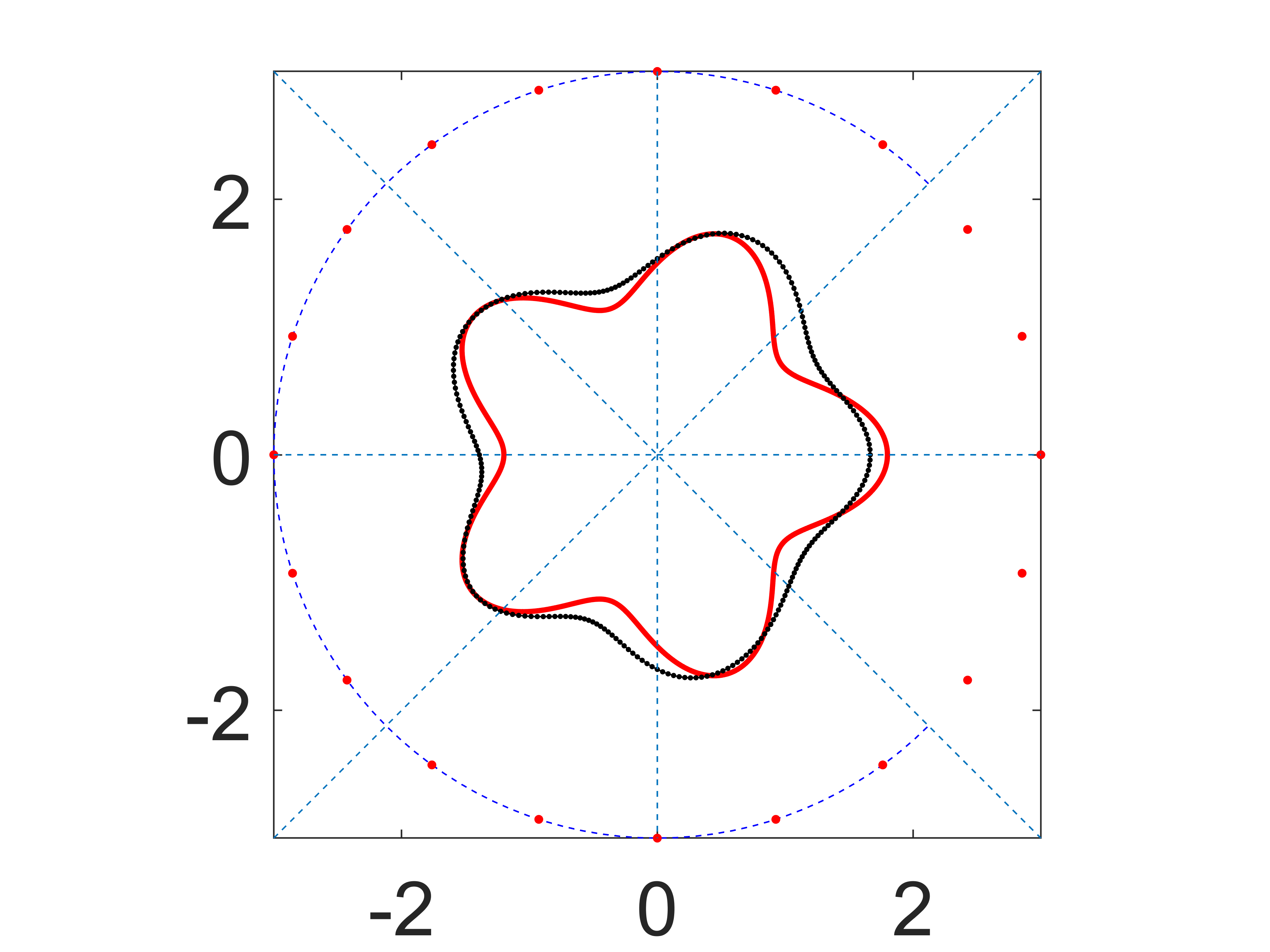}}
	\subfigure[]{\includegraphics[width=0.32\textwidth]{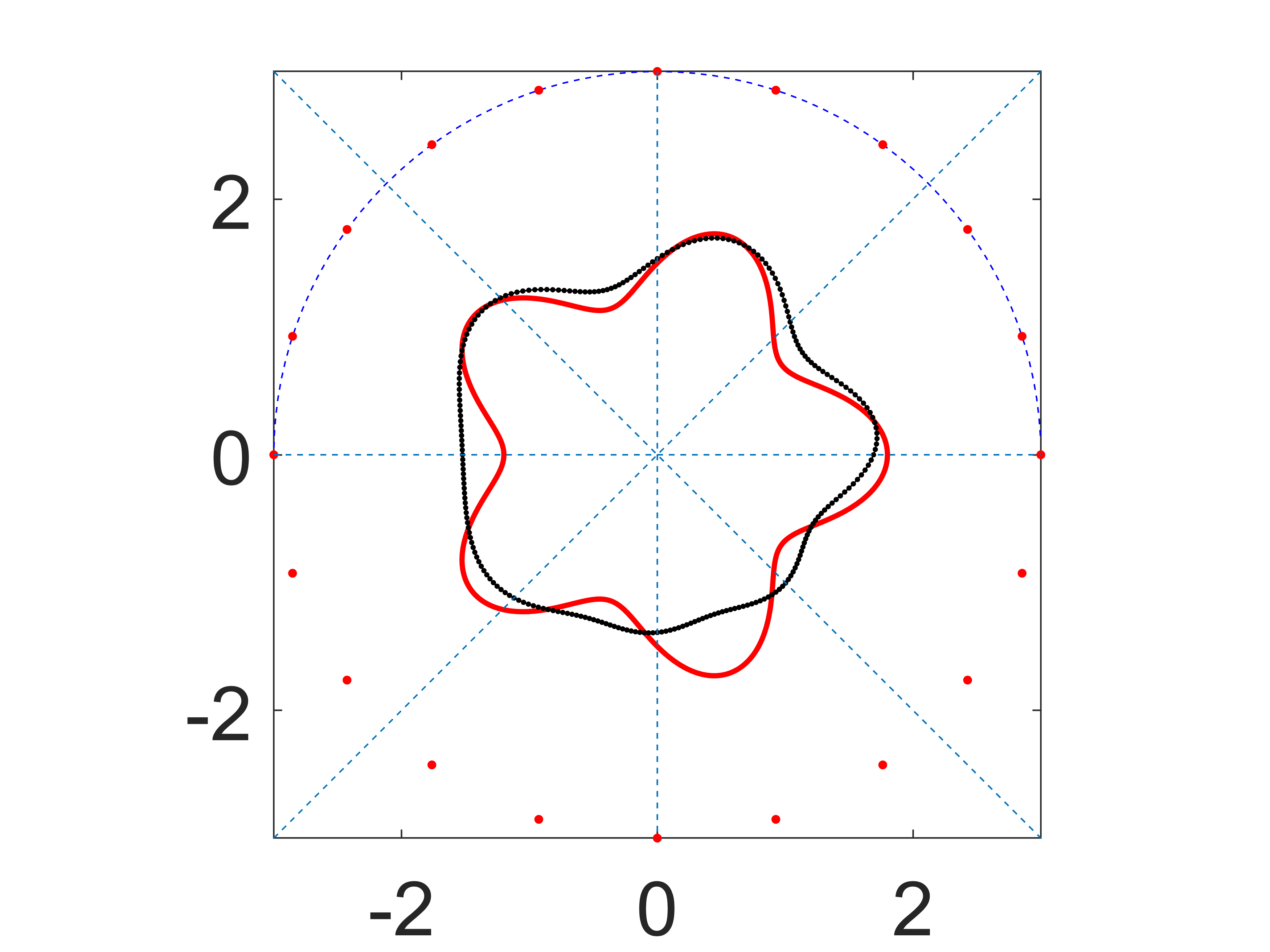}}
	\subfigure[]{\includegraphics[width=0.32\textwidth]{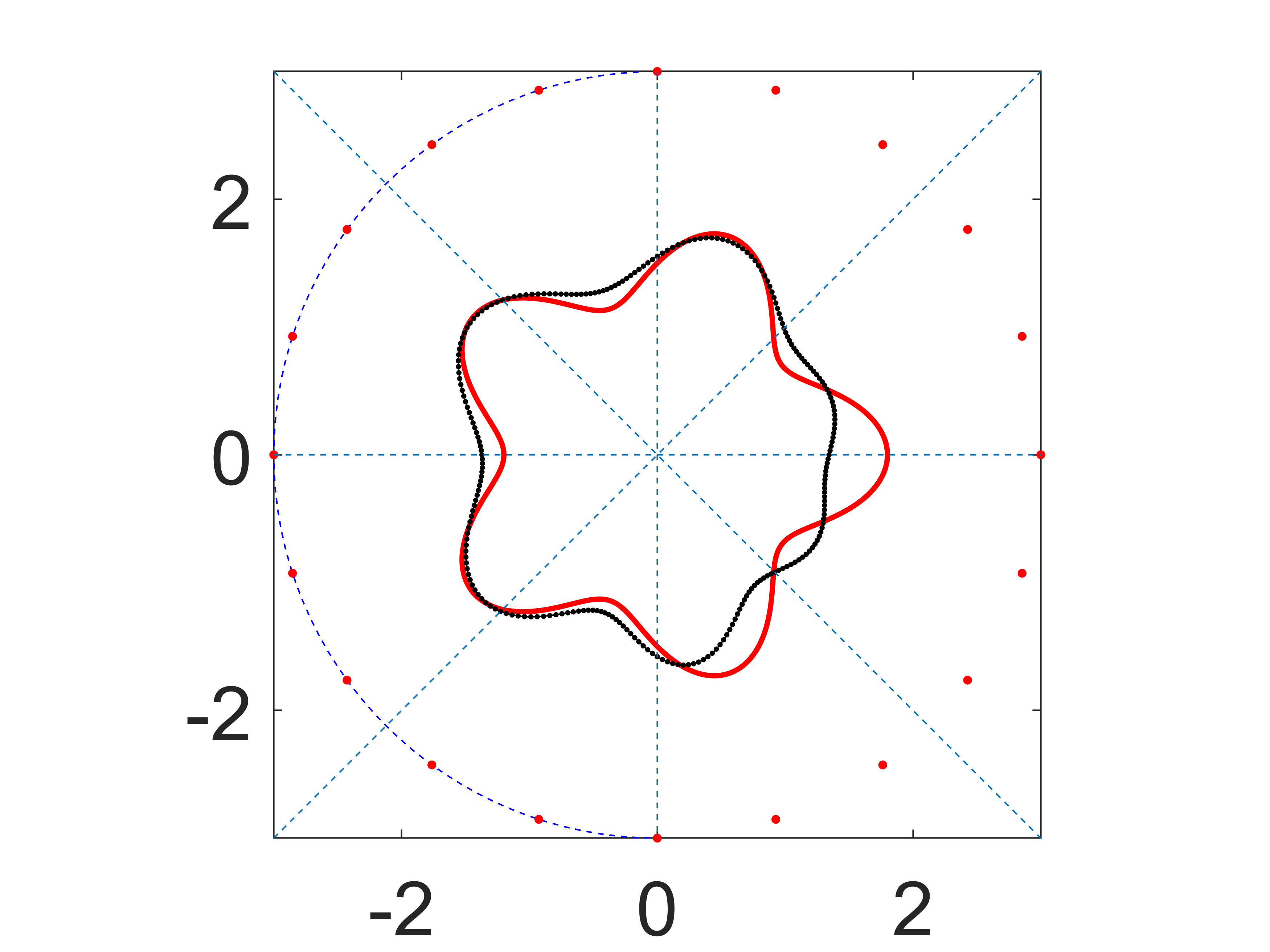}}
	\caption{Reconstruction of the starfish-shaped obstacle under different observation apertures. (a) exact obstacle and initial guess; (b) $\theta_i\in[0,2\pi);$ (c) $\theta\in[{\pi}/{2},2\pi);$ (d) $\theta\in[{\pi}/{4},{7\pi}/{4});$ (e) $\theta\in[0,\pi);$ (f) $\theta\in[{\pi}/{2},{3\pi}/{2}).$}\label{fig:starfish}
\end{figure}

\begin{figure}
	\centering  
	\subfigure[]{\includegraphics[width=0.24\textwidth]{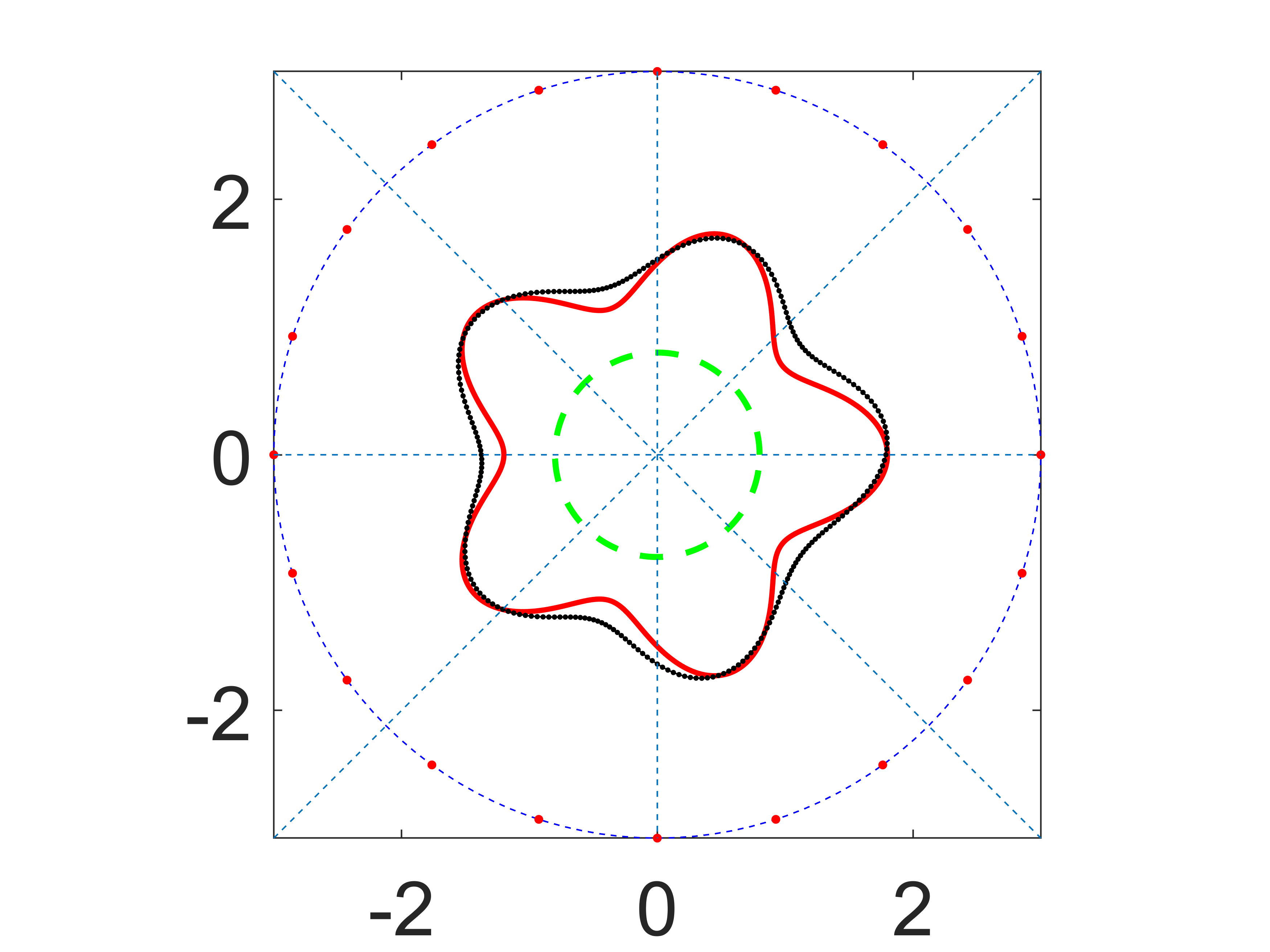}}
	\subfigure[]{\includegraphics[width=0.24\textwidth]{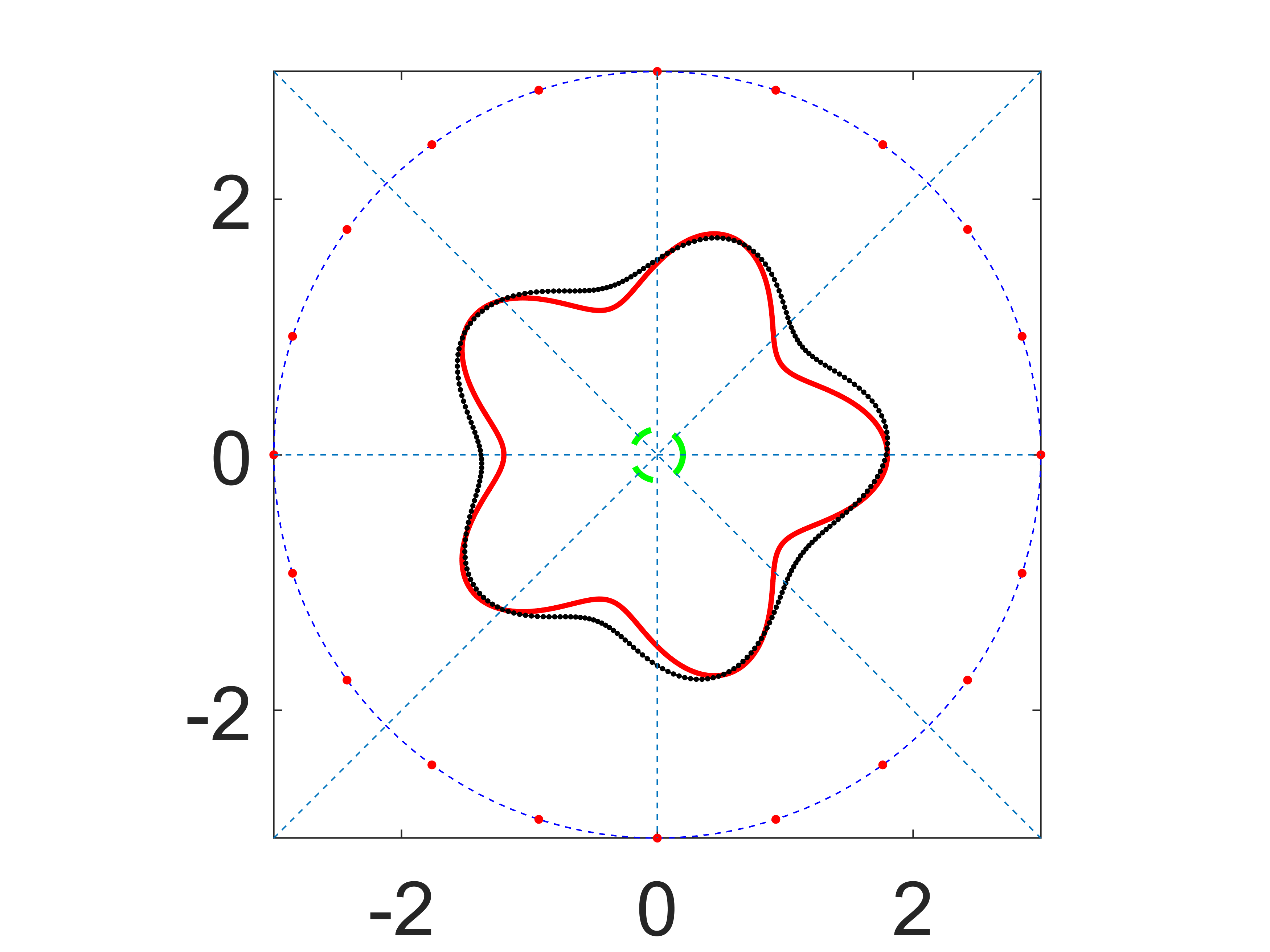}}
	\subfigure[]{\includegraphics[width=0.24\textwidth]{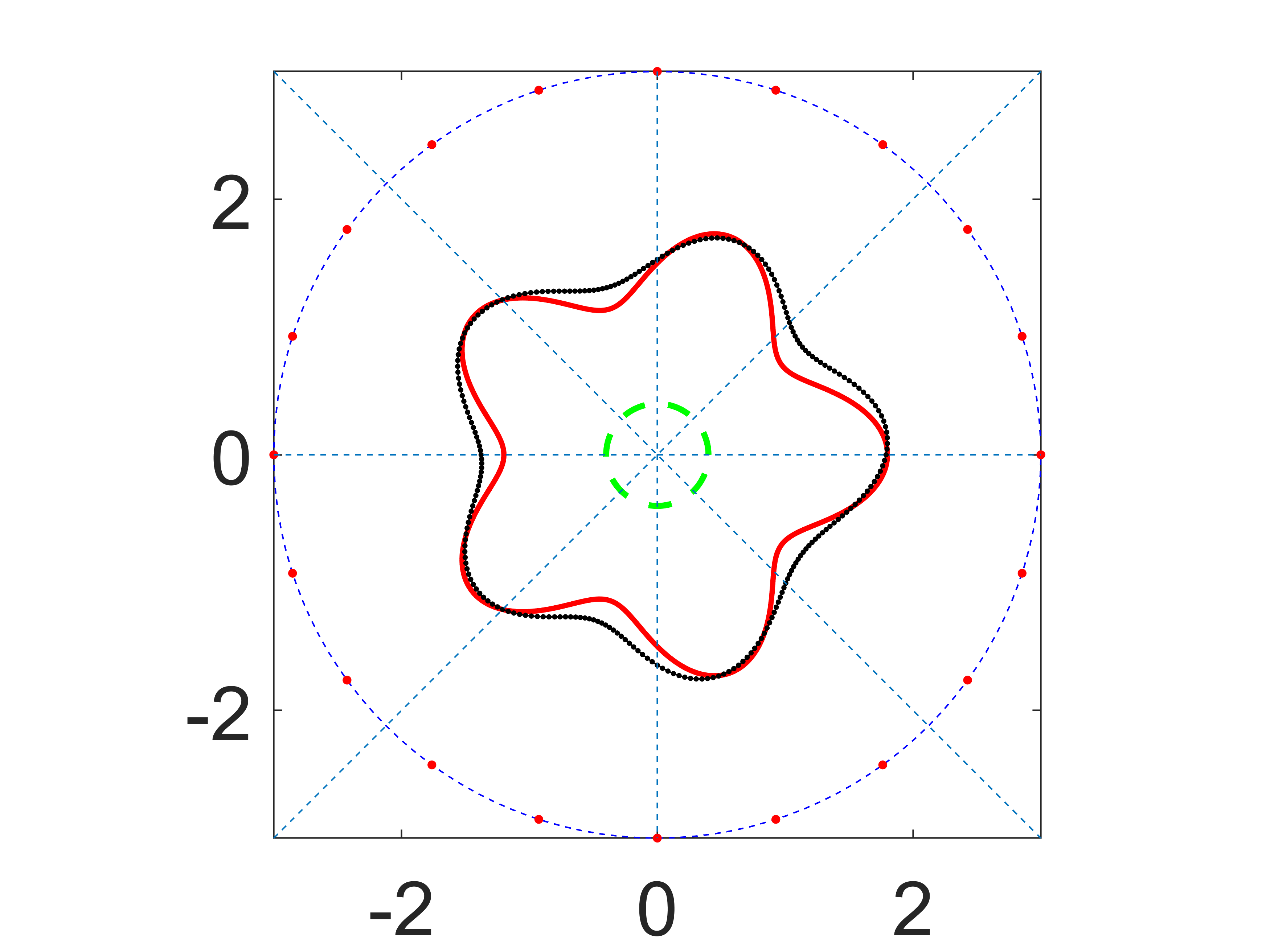}}
	\subfigure[]{\includegraphics[width=0.24\textwidth]{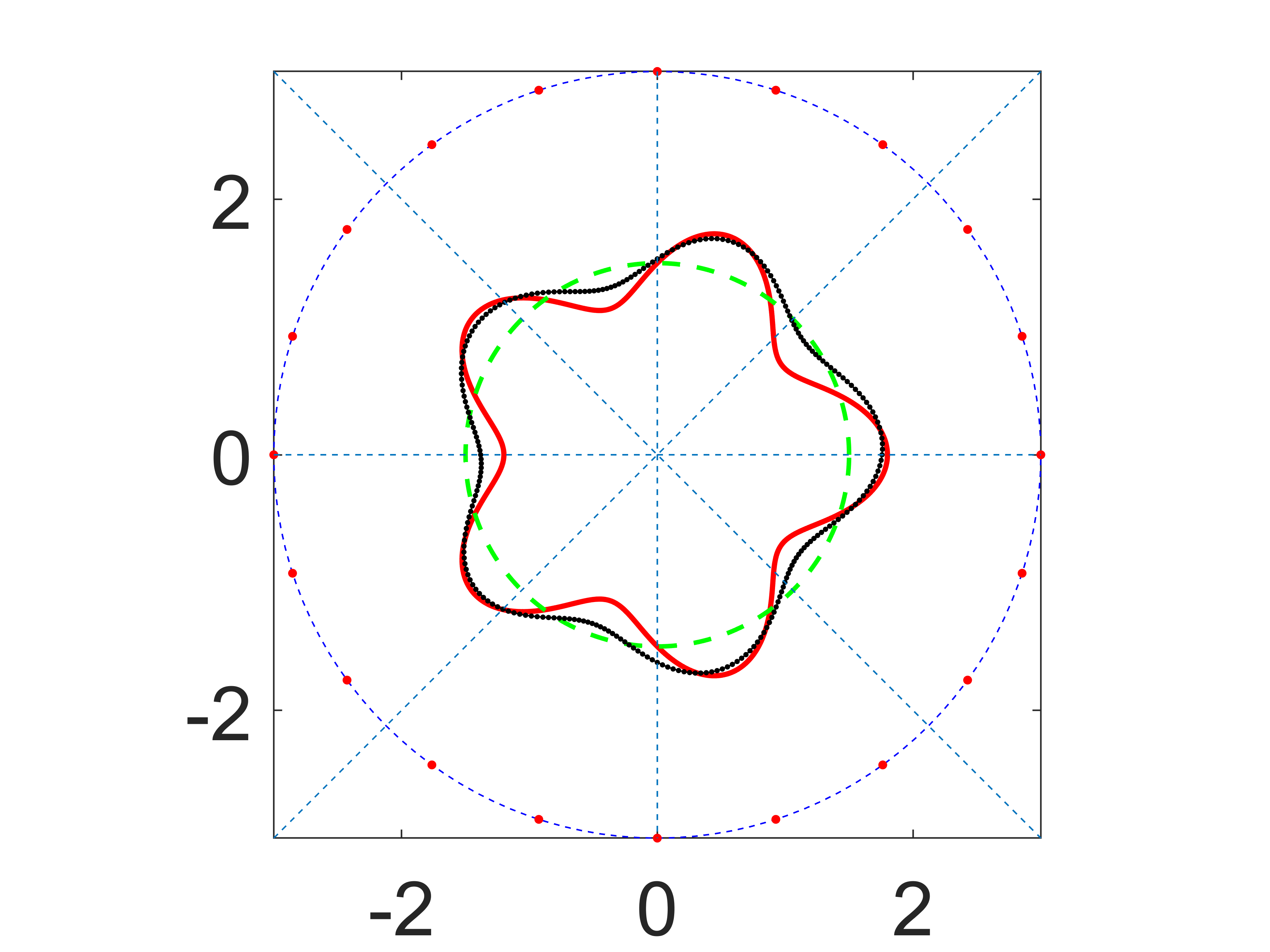}}
	\caption{Reconstruction of the starfish-shaped obstacle subject to different initial guesses (green dashed line: initial guess).}\label{fig:starfish2}
\end{figure}
\end{example}
	
\section{Conclusion}\label{sec:conclusion}

A novel Newton-type method is proposed to solve the inverse elastic scattering problem. Based on the Helmholtz decomposition and the Fourier expansion, we develop an approximate of the scattered field and establish the error estimates mathematically. Further, a Newton-type method is proposed for the inverse scattering problem, where the derivative can be explicitly computed once the approximate scattered elastic field is derived based on the former step. The method is efficient and no forward solver is involved in the inversion process. In addition, we carry on theoretical analysis to study the convergence property of the Newton method based on the error estimates. Numerically, we test the performance of the novel method and extend the method to the partial aperture problem.

In future work, we shall consider extending this Newton method to other physical models such as the inverse acoustic-elastic interaction problem and the bi-harmonic inverse scattered problem. We believe that the idea proposed here is applicable to these problems. We would also develop the corresponding theoretical analysis concerning the limited-aperture problem. 
	
	
\section*{Appendix}
\noindent\appendix 
Let us denote
\begin{align*}
	S=\sum_{n>N}n^2\tau_1^{-2n}.
\end{align*}
Then,
\begin{align*}
	(1-\tau_1^{-2})S&=(N+1)^2\tau_1^{-2(N+1)}+\sum_{n>N}(n+1)^2\tau_1^{-2(n+1)}-\sum_{n>N}n^2\tau_1^{-2(n+1)}\\
	&=(N+1)^2\tau_1^{-2(N+1)}+\sum_{n>N}(2n+1)\tau_1^{-2(n+1)}.
\end{align*}
Denote
\begin{align}\label{eq:a3}
	S_1=(1-\tau_1^{-2})S-(N+1)^2\tau_1^{-2(N+1)}=\sum_{n>N}(2n+1)\tau_1^{-2(n+1)},
\end{align}
then $S_1=\tau_1^{-2N-2}S_2,$
with $$S_2=\sum_{n=1}^\infty(2N+2n+1)\tau_1^{-2n}.$$
It is easy to check that
\begin{align*}
	(\tau_1^2-1)S_2=2N+3+2\sum_{n=1}^\infty\tau_1^{-2n}=2N+3+\dfrac{2}{\tau_1^2-1},
\end{align*}
i.e.,
\begin{equation*}
	S_2=\dfrac{2N+3}{\tau_1^2-1}+\dfrac{2}{(\tau_1^2-1)^2}.
\end{equation*}
Further, it holds that
\begin{align}\label{eq:a5}
	S_1=\tau_1^{-2N-2}S_2=(2N+3)\tau_1^{-2N}(\tau_1^2-1)^{-2}-(2N+1)\tau_1^{-2N-2}(\tau_1^2-1)^{-2}.
\end{align}

Combining \eqref{eq:a3} and \eqref{eq:a5} gives
\begin{align*}
	S&=(1-\tau_1^{-2})^{-1}\Big(S_1 + (N+1)^2\tau_1^{-2(N+1)}\Big)\\
	&=\dfrac{(2N+3)\tau_1^2-(2N+1)+(N+1)^2(\tau_1^2-1)^2}{(\tau_1^2-1)^3}\tau_1^{-2N}\\
	&\le \tau_1^{-2N}(\tau_1^2-1)^{-3}\Big((2N+3)\tau_1^2+(N+1)^2(\tau_1^2-1)^2\Big)\\
	&=\tau_1^{-2N}(\tau_1^2-1)^{-3}\Big((N+1)^2\tau_1^4-(2N^2+2N-1)\tau_1^2+(N+1)^2\Big).
\end{align*}
Noticing $\tau_1>1,$ we know that
\begin{align*}
	-(2N^2+2N-1)\tau_1^2+(N+1)^2&\le\tau_1^2\left((N+1)^2-(2N^2+2N-1)\right)\\
	&\le\tau_1^2(2-N^2).
\end{align*}
Thus, for $N\ge 2,$
\begin{align*}
	S\le\tau_1^{-2N}(\tau_1^2-1)^{-3}(N+1)^2\tau_1^4\le4N^2\tau_1^{4-2N}(\tau_1^2-1)^{-3},
\end{align*}
which completes the proof.


\footnotesize
\bibliographystyle{plain}

\end{document}